\documentclass[11pt,reqno]{amsproc}
\usepackage[colorlinks,linkcolor=Green,citecolor=blue]{hyperref}
\usepackage{lineno}
\modulolinenumbers[5]
\usepackage{times}
\usepackage{amsmath}
\usepackage{amssymb}
\usepackage{mathrsfs}
\usepackage{amsthm}
\usepackage{bm}
\usepackage{graphicx}
\usepackage{algorithm} 
\usepackage{algpseudocode}
\usepackage{color}
\definecolor{Green}{RGB}{0,128,0}
\usepackage{subfigure}
\usepackage{enumitem}
\usepackage{float}
\usepackage{geometry}
\newtheorem{definition}{Definition}[section]
\newtheorem{lemma}[definition]{Lemma}
\newtheorem{hypothesis}[definition]{Hypothesis}
\newtheorem{assumption}{Assumption}
\newtheorem{corollary}[definition]{Corollary}
\newtheorem{theorem}[definition]{Theorem}
\newtheorem{example}[definition]{Example}
\newtheorem{proposition}[definition]{Proposition}
\newtheorem{remark}[definition]{Remark}

\usepackage{color,xcolor}
\definecolor{Green}{RGB}{46 139 87}

\newcommand{\ud}{\mathrm d}
\newcommand{\R}{\mathbb{R}}
\numberwithin{equation}{section}

\newcommand{\E}{\mathbb{E}}\allowdisplaybreaks[4]

\begin{document}

\title[LDP and MDP for SPDE on graph]{Large and moderate deviation principles for stochastic partial differential equation on graph}

\subjclass[2010]{60H15, 35R02, 60F10, 60F05}



\author{Jianbo Cui}
\address{Department of Applied Mathematics, The Hong Kong Polytechnic University, Hung Hom, Kowloon, Hong Kong}
\email{jianbo.cui@polyu.edu.hk}

\author{Derui Sheng}
\address{Department of Applied Mathematics, The Hong Kong Polytechnic University, Hung Hom, Kowloon, Hong Kong}
\email{sdr@lsec.cc.ac.cn (Corresponding author)}

\thanks{This work is supported by MOST National Key R\&D Program No. 2024FA1015900, the Hong Kong Research Grant Council GRF grant 15302823, NSFC/RGC Joint Research Scheme N$\_$PolyU5141/24, NSFC grant 12301526, internal funds (P0041274, P0045336) from Hong Kong Polytechnic University, and the
CAS AMSS-PolyU Joint Laboratory of Applied Mathematics.
}

\keywords{Stochastic partial differential equation, metric graph, large deviation principle, moderate deviation principle}

\begin{abstract}
In this paper, we study large and moderate deviation principles for stochastic partial differential equations (SPDEs) on metric graphs and their associated multiscale models via the weak convergence approach, providing a refined characterization of the probabilities of rare events. Several challenges unique to the graph setting are encountered, including operator degeneracy near vertices and the lack of compactness on non-compact graphs. To address these difficulties, we introduce novel weighted Sobolev spaces on graphs, and prove compact embedding results specifically adapted to the degeneracy structure. Our analysis is particularly applicable to SPDEs on graphs arising as limits of stochastic reaction-diffusion systems on narrow domains and from fast-flow asymptotics of stochastic incompressible fluids, yielding new deviation results for these models.
\end{abstract}

\maketitle
\section{Introduction}
Stochastic partial differential equations (SPDEs) on graphs, particularly on metric graphs, 
 have emerged as fundamental models in neuroscience, physics, and network dynamics. They provide a continuum framework for SPDEs on discrete graphs (see, e.g., \cite{MR4612606,CLZ23}), and describe the transmission of electrical signals along dendritic trees, with stochastic impulsive inputs capturing the influence of excitatory and inhibitory signals from neighboring neurons \cite{BMZ08,BM10}. 
The structural flexibility of metric graphs allows the incorporation of nontrivial topologies and diverse boundary conditions at vertex sets. Such equations also arise in various other applications, including free-electron models for organic molecules \cite{LP36}, superconductivity in granular and engineered materials \cite{AS83}, wave propagation in acoustic and electromagnetic networks \cite{CRH87}, Anderson transition in disordered wire \cite{AMR20,SB82}, quantum chaos \cite{BK13}, statistical modeling \cite{BSW24} and more.

Recently, SPDEs on graphs have also provided an effective framework for describing the asymptotic behavior of SPDEs defined on Euclidean domains (see, e.g. \cite{CF17,CF19}). To illustrate, consider the following SPDE on a domain $D\subset \R^2$
\begin{equation}\label{eq:intro2}
\partial_t u_\delta^\epsilon(t,x)=L_{\delta} u_\delta^\epsilon(t,x)
 +b(u_\delta^\epsilon(t,x))+\sqrt{\epsilon}g(u_\delta^\epsilon(t,x))\partial_t\mathcal{W}(t,x)
\end{equation}
for $t\in(0,T]$ and $x\in D$,
where $\mathcal{W}$ is a $Q$-Wiener process, the functions $b, g:\R\to\R$ are Lipschitz continuous, and the parameters $\epsilon,\delta>0$. The dominated operator $\{L_{\delta}\}_{\delta>0}$ in \eqref{eq:intro2} is assumed to be the infinitesimal generator of a fast diffusion process $\{X_\delta(\cdot)\}_{\delta>0}$. 
We are interested in the regime where, asymptotically,
the slow motion $\mathcal{H}(X_\delta)$ converges to a diffusion process $\bar{Y}$ on a graph
 $\Gamma$ as $\delta\to0$, where 
 $\mathcal{H}$ is a mapping from $D$ to $\R$; see section \ref{S:Pre} for concrete examples of $X_\delta$ and $\mathcal{H}$.
The graph $\Gamma$ is constructed by identifying the points of each connected component of the level set $C(z):=\{x\in D:\mathcal{H}(x)=z\}$ for $z\in\R$.
 For a fixed $\epsilon>0$, the small $\delta$ limit of \eqref{eq:intro2} leads to the following
 SPDE on graph 
\begin{equation}\label{eq:intro}
\partial_t \bar{u}^\epsilon(t, z, k)=\bar{L} \bar{u}^\epsilon(t, z, k)+b(\bar{u}^\epsilon(t, z, k))+\sqrt{\epsilon} g(\bar{u}^\epsilon(t, z, k)) \partial_t \bar{\mathcal{W}}(t, z, k)\end{equation}
for $t\in(0,T]$ and $(z,k)\in\Gamma$,
where $\bar{\mathcal{W}}$ is the formal projection of $\mathcal{W}$ on graph (see \eqref{eq:barW}). The graph
 $\Gamma$, equipped with the shortest path metric, becomes a metric space, which, together with the infinitesimal generator $\bar{L}$ of the diffusion process $\bar{Y}$, constitutes a quantum graph 
 (cf.~\cite{BKKS24}). Under Kirchhoff-type gluing conditions at the vertices, \eqref{eq:intro} admits a unique solution $\bar{u}^\epsilon$ in $\mathcal{C}([0,T];\bar{H}_\gamma)$, where $\bar{H}_\gamma$ denotes a weighted $L^2$-space on graph $\Gamma$ associated with a suitable graph weight function $\gamma:\Gamma\to (0,\infty)$ (see Assumption \ref{asp:gamma}). 

The primary goal of this paper is to establish the large and moderate deviation principles (LDPs and MDPs) of the SPDE \eqref{eq:intro} on graph and its related multiscale model \eqref{eq:intro2}.
The LDP and MDP provide a refined characterization of the probabilities of rare events, complementing the law of large numbers and central limit theorem. On one hand, the LDP concerns the exponential decay rates of probabilities of rare events, and has found broad applications in areas such as thermodynamics, statistics, information theory, and engineering \cite{FW12,GLL24,MM20,RW21,WWZ24}. On the other hand, the MDP
describes deviations for centered and rescaled processes at intermediate scales, offering valuable information about convergence rates and serving as a tool for constructing asymptotic confidence intervals; see, e.g., \cite{GZ11,KW83} and references therein.

While significant progress has been made on LDPs and MDPs for SPDEs on Euclidean domains (cf. \cite{BM12,HLL21} and references therein), results for metric graphs remain relatively scarce 
\cite{CH24}. Understanding these deviation principles from Euclidean domains to the setting of graphs introduces several additional challenges.
 First, the operator $\bar{L}$ may be degenerate
near vertices of the graph (see \eqref{eq:barL}). Second, 
 the standard Sobolev embedding $W^{1,2}\hookrightarrow L^2$ may fail to be compact on non-compact graph (see Proposition \ref{lem:counterexample} for more details). Both issues complicate the proof of compactness for the level sets of the rate functions. 
Third, for the deviation principle estimates of \eqref{eq:intro2}, additional difficulties arise in controlling the asymptotics of $L_\delta$ as $\delta\to0$. Last but not least, in contrast to the LDP of \eqref{eq:intro}, studying its MDP requires the weighted \( L^p \)-regularity ($p\ge1$) estimates for the related skeleton equation whose dominated operator $\bar{L}$ is not uniformly elliptic.

To address these challenges, we first introduce another graph weight function \( \gamma_1: \Gamma \to (0, \infty) \) such that \eqref{eq:intro} remains well-posed in \( \mathcal{C}([0,T]; \bar{H}_{\gamma_1}) \), and 
the associated Sobolev space $\bar{\mathbb{W}}^{1,2}_{\gamma_1}$ (see \ref{eq:Lambda})
is compactly embedded into \( \bar{H}_\gamma \) (see Assumption \ref{asp:compact}). This construction allows us to show that the rate functions possess compact level sets and 
are therefore good rate functions. Furthermore, 
 we propose a framework for establishing both LDP and MDP of \eqref{eq:intro} on $\mathcal{C}([0,T];\bar{H}_\gamma)$ via the weak convergence approach and semigroup method (see Theorems \ref{thm:LDP-graph} and \ref{thm:MDP-graph}). 
By analyzing the asymptotic of $L_{\psi(\epsilon)}$, 
 we further demonstrate that the LDP also holds for the family $\{u_{\psi(\epsilon)}^\epsilon\}_{\epsilon>0}$ as $\epsilon \to 0$, provided that $\lim_{\epsilon \to 0}\psi(\epsilon) =0$ (see Theorem \ref{thm:LDP}). For comparison, see \cite{CH24} for a previous result on the LDP of the model in \cite{CF17} with additive noise and $\delta=\epsilon^{\frac 1\kappa}$ for some $\kappa>0$.
Moreover, we find that the convergence of the multiscale model \eqref{eq:intro2} to the SPDE on graph \eqref{eq:intro} is consistent in the small-noise limit, in the sense that the LDPs for the SPDE on the domain $D$
coincides with those for the corresponding SPDE on the graph $\Gamma$ with small noise.
The analysis of the MDP for \eqref{eq:intro2} with $\delta=\psi(\epsilon)$, however, requires a precise understanding of the convergence rate of \eqref{eq:intro2} to \eqref{eq:intro}, which remains unclear and is left for future investigation.

 Our framework and main results on large and moderate deviations are formulated in section \ref{S:MR}, while their proofs are collected in section \ref{S:Pf}.
In section \ref{S:ND}, we focus on an SPDE on a compact graph \( \Gamma \), which originally appeared in \cite{CF17} as the limiting equation of stochastic reaction-diffusion equations on narrow domains \( D_\delta := \{x\in \R^2,(x_1, x_2/\delta) \in D\} \), as the width parameter \( \delta \to 0\).
Section \ref{S:FA} investigates an SPDE on an unbounded graph $\Gamma$, first introduced in \cite{CF19}. This model serves as the limiting equation for the evolution of particle density in $D=\R^2$, driven by an incompressible flow with the stream function $\mathcal{H}$, in the regime where the scaling factor $1/\delta$ of the divergence-free advection tends to infinity. Our main efforts lie in identifying
 suitable $\gamma_1$ and proving the compactness of the embedding $\bar{\mathbb{W}}^{1,2}_{\gamma_1}\hookrightarrow \bar{H}_{\gamma}$. These are achieved through a
 careful analysis of the asymptotic properties of the coefficients $\alpha_k$ and $T_k$ in \eqref{eq:barL} and the application of embedding theorems for weighted Sobolev spaces (see, e.g., \cite{CRW13,GU09}) adapted to the degeneracy structure of the problem.

\section{Preliminaries}\label{S:Pre}
This section collects the preliminaries for the SPDE on the graph \eqref{eq:intro} and its associated multiscale model \eqref{eq:intro2}.
In the sequel,
for a measure space $(\mathbb{M},\mathscr{B}(\mathbb{M}), \tilde{\mu})$ and $1\le p<\infty$, we denote by $L^p(\mathbb{M},\ud\tilde{\mu})$ the space of measurable functions $F:\mathbb{M}\to \R$ with
$\int_{\mathbb{M}}|F(x)|^p\tilde{\mu}(\ud x)<\infty.$ When $\mathbb{M}\subset\R^d$ with some $d\in\mathbb{N}^+$ and $ \tilde\mu(\ud x)=\tilde{w}(x)\ud x$ for a weight function $\tilde{w}:\mathbb{M}\to\R$, we also use the notation $L^p(\mathbb{M},\tilde{w}\ud x)$ or $L^p(\mathbb{M},\tilde{w}(x)\ud x)$. In particular, if $\mathbb{M}\subset\R^d$ and $\tilde{\mu}$ is the Lebesgue measure, we abbreviate $L^p( \mathbb{M},\ud\tilde{\mu})$ as $L^p( \mathbb{M})$, and denote by $W^{k,p}(\mathbb{M})$, $k\in \mathbb{N}^+$, the standard Sobolev space. For convenience, we also use the abbreviations $L^p(a,b)$ and $W^{k,p}(a,b)$ if $\mathbb{M}=(a,b)\subset\R$ is an open interval and $\tilde{\mu}$ is the Lebesgue measure, where $-\infty \le a<b\le \infty$.
Henceforth, given two Hilbert spaces $V_1$ and $V_2$, we write \( \mathcal{L}_2(V_1, V_2) \) (resp. \( \mathcal{L}(V_1, V_2) \)) for the space of Hilbert–Schmidt operators (resp. bounded linear operators) from \( V_1 \) to \( V_2 \), and set \( \mathcal{L}(V_1):=\mathcal{L}(V_1,V_1) \).

We start with the basic setting for a family of diffusion processes $\{X_\delta(t),t\in[0,T]\}_{\delta>0}$, where each $X_\delta(t)$ takes values in an open set $D\subset\R^2$. Such multiscale diffusion processes typically arise from studying averaging principles for systems with conservative laws perturbed by small noise, under a change of time scale (see, e.g., \cite{FW12}).
We are interested in the case where, asymptotically, the transformed process
$\{\mathcal{H}(X_\delta(t)),t\in[0,T]\}$ 
reduces to some stochastic process on a graph as $\delta\to 0$ for some function $\mathcal{H}:D\to \R$.
 Let $\{\textup B(t)\}_{t\ge0}$ be a $2$-dimensional Brownian motion defined on a filtered probability space $(\Omega, \mathcal{F},\{\mathcal{F}_t\}_{t \geq 0}, \mathbb{P})$.
 Two concrete examples of $X_\delta$ and $\mathcal{H}$ are given as follows, which will be revisited in sections \ref{S:ND} and \ref{S:FA}, respectively.
\begin{example}\cite{CF17}\label{Ex:1}
Let $D=B_1(0)=\{x\in\R^2:\|x\|<1\}$ be the unit disk. 
For each $\delta>0$, consider the following diffusion process with reflecting boundary conditions 
\begin{equation}\label{eq:X-ND}
\ud X_\delta(t)=\sqrt{\sigma_\delta}\ud \textup B(t)+\sigma_\delta \textup{\textbf{n}}(X_\delta(t))\ud \mathrm{L}^\delta(t),\quad X_\delta(0)=x_0\in D,
\end{equation}
where $\textup{\textbf{n}}(x)=(\textup{\textbf{n}}_1(x),\textup{\textbf{n}}_2(x))$ denotes the unit inward normal vector at $x\in \partial D$, $\mathrm{L}^\delta(t)$ is the local time of the process $X_\delta(t)$ on $\partial D$, and 
$\sigma_\delta:=\left(\begin{array}{cc}1 & 0 \\0 & \delta^{-2}\end{array}\right).$
We define $\mathcal{H}:D\to\R$ by $\mathcal{H}(x)=x_1$ for any $x=(x_1,x_2)\in D$. We will revisit this example under a more general domain $D$ in subsection \ref{S:NDX}.
\end{example}
\begin{example}\cite{CF19}\label{Ex:2}
For each $\delta>0$, consider the following diffusion process 
\begin{equation}\label{eq:FAXt}
\ud X_\delta(t)=\frac{1}{\delta} \nabla^{\perp}\mathcal{H} \left(X_\delta(t)\right) \ud t+\ud \textup{B}(t),\quad X_\delta(0)=x_0\in D=\R^2,
\end{equation}
where $\nabla^{\perp} =J\cdot\nabla$ with $J=\left(\begin{array}{cc}0 & -1 \\1 & 0\end{array}\right)$ and $\nabla$ being the gradient operator. 
General conditions on the Hamiltonian $\mathcal H$ will be specified in subsection \ref{S:FAX}; as an illustration, one may take
$\mathcal{H}(x)=\|x\|^2+\sqrt{1+\|x\|^2}-1$ for $x\in D$.

\end{example}
\subsection{Diffusion process on graph}
Given a function $\mathcal{H}:D\to \R$, a
 graph $\Gamma$ can be constructed by identifying the points of each connected component of each level set $C(z):=\{x\in D:\mathcal{H}(x)=z\}$ of $\mathcal{H}$. Assume that $\Gamma$ comprises a finite set of vertices $\{O_i\}_{i=1}^{m_1}$ and a finite set of edges $\{I_k\}_{k=1}^m$. A natural projection $\Pi:D\to \Gamma$ is defined as $$\Pi(x)=(\mathcal{H}(x),k(x)),\quad x\in D,$$
 where $k(x)$ is the index of the edge $I_{k(x)}$ containing $\Pi(x)$. Each edge $I_k$ on the graph $\Gamma$ connecting two vertices $O_{k_1}$ and $O_{k_2}$ with $\mathcal{H}(O_{k_1})<\mathcal{H}(O_{k_2})$ can be identified with an interval $(a_k,b_k):=(\mathcal{H}(O_{k_1}),\mathcal{H}(O_{k_2}))\subset\R$, denoted as $I_k\cong(a_k,b_k)$.
 This identification allows any point on $\Gamma$ to be represented by a coordinate $(z,k)$, where $z\in[a_k,b_k]$ and $k\in\{1,\ldots,m\}$. Equipped with the shortest path metric,
 $\Gamma$ becomes a metric space, where the distance between two points is defined as
the length of the shortest path on $\Gamma$ connecting them (see, e.g., \cite{BKKS24}).

Assume that the projected process $\Pi(X_\delta(\cdot))$
 converges, in the sense of weak convergence of distributions in $\mathcal C([0,T];\Gamma)$, to a limiting process $\bar{Y}$ on $\Gamma$.
The limiting process $\bar{Y}$ is a Markov process on the graph $\Gamma$ characterized by its infinitesimal generator
\begin{equation}\label{eq:barL}
\bar{L} f(z, k)=\frac{1}{ 2T_k(z)} \frac{\ud}{\ud z}\left(\alpha_k \frac{\ud f}{\ud z}\right)(z), \quad\text{ if $(z, k)$ is an interior point of $I_k$,}
\end{equation}
where the coefficients are given by
\begin{equation}\label{eq:AT}
\alpha_k(z)=\oint_{C_k(z)} |\nabla \mathcal{H}(x)| \ud l_{z, k},\quad T_k(z)=\oint_{C_k(z)} \frac{1}{|\nabla \mathcal{H}(x)|}\ud l_{z,k}.
\end{equation}
Here $\ud l_{z,k}$ is the length element on the connected component $C_k(z)$ of the level set $C(z)$, corresponding to the edge $I_k$, and
$T_k(z)$ is the
 period of the motion on $C_k(z)$.
The domain Dom$(\bar{L})$ of $\bar{L}$ consists of 
continuous functions $f$ on the graph $\Gamma$ that are twice continuously differentiable in the interior of each edge such that for any vertex $O_i=(z_i,k_{i_j})$ with $j=1,2,\ldots,l$, the limits
$$\lim_{(z,k_{i_j})\to O_i}\bar{L} f(z,k_{i_j}),\quad j=1,2,\ldots, l,$$
exist, are finite, and are independent of the choice of incident edge $I_{k_{i_j}}$. Moreover, at each interior vertex $O_i$, the following gluing condition is satisfied:
\begin{equation}\label{eq:glue}
\sum_{j=1}^l\alpha_{k_{i_j}}(z_i) d_{k_{i_j}}f(z_i,k_{i_j})=0,
\end{equation}
where $d_{k_{i_j}}$ is the differentiation along $I_{k_{i_j}}$ and the sign $+$ is taken if the $\mathcal{H}$-coordinate increases along $I_{k_{i_j}}$ and the sign $-$ is taken otherwise (see also \cite{CF19} for more details).
 The infinitesimal generator $\bar{L}$ of the limiting process $\bar{Y}$ serves as the dominated generator of the parabolic SPDE \eqref{eq:intro} on graph. 

\subsection{Projection and function spaces}
For simplicity, we define the pullback of a function $f:\Gamma\to\R$ via
$
f^{\vee}:=f\circ\Pi
$.
It should be noted that $f^{\vee}$ is a well-defined function on $D$ provided that $f$ is continuous on vertices, i.e., $f(z,k_i)=f(z,k_j)$ whenever $(z,k_i)$ and $(z,k_j)$ correspond to the same interior vertex of $\Gamma$.
For a function $\varphi:D\to \R$, we define its projection $\varphi^\wedge:\Gamma\to \R$ via the averaging over the connected components of level sets of $\mathcal{H}$, namely, for $(z,k)\in\Gamma$,
\begin{equation}\label{eq:wedge}\varphi^\wedge(z,k):=\oint_{C_k(z)} \varphi(x)\ud \mu_{z,k}\quad\text{with}\quad\ud\mu_{z,k}=\frac{1}{T_k(z)} \frac{1}{|\nabla \mathcal{H}(x)|}\ud l_{z,k}.
\end{equation}
The measure $\mu_{z,k}$
is a probability measure supported on $C_k(z)$.
 Note that $(f^{\vee}\varphi)^{\wedge}=f\varphi^\wedge$ for any $f:\Gamma\to \R$ and $\varphi:D\to \R$.
 Given a graph weight function $\tilde{\gamma}:\Gamma\to(0,\infty)$, we denote 
 $H_{\tilde{\gamma}} :=L^2(D,\tilde{\gamma}^\vee(x)\ud x)$ and $\bar{H}_{\tilde{\gamma}}:=L^2(\Gamma,\ud\nu_{\tilde{\gamma}})$, where $\nu_{\tilde{\gamma}}$ is the pushforward measure of ${\tilde{\gamma}}^\vee(x) \ud x$ by $\Pi$, i.e.,
 for any $A\in \mathscr{B}(\Gamma)$,
\begin{align}\label{eq:nugamma}
\nu_{\tilde{\gamma}}(A)=\int_{\Pi^{-1}(A)}\tilde{\gamma}^\vee(x)\ud x
&=\sum_{k=1}^m\int_{I_k}\oint_{C_k(z)}\mathbb{I}_{A}(\Pi(x))\tilde{\gamma}(\Pi(x))\ud x\\\notag
&=\sum_{k=1}^m\int_{I_k} \mathbb{I}_{A}(z,k)\tilde{\gamma}(z,k)T_k(z)\ud z,
\end{align}
 where we have used \eqref{eq:AT} and $\ud x=\frac{1}{|\nabla\mathcal{H}(x)|}\ud l_{z,k}\ud z$ (see \cite[formula (2.5)]{CF19}).
Here, $\mathbb{I}_{A}$ denotes the indicator function on the set $A$. 
Let $\bar{\mathbb{W}}^{1,2}_{\tilde{\gamma}}$ be the space of all functions $f\in \bar{H}_{\tilde{\gamma}}$ such that the weak derivative of $f$ along each edge $I_k$, which we write as $f^\prime(\cdot,k)$ or $\frac{\ud}{\ud z}f(\cdot,k)$, belongs to $L^2(I_k,\alpha_k\tilde{\gamma}\ud z)$.
This space is equipped with the norm
\begin{equation}\label{eq:Lambda}
\|{ f}\|_{\bar{\mathbb{W}}^{1,2}_{\tilde{\gamma}}}:=\left(\sum_{k=1}^m\int_{I_k}\left[|f(z,k)|^2T_k(z)+|f^\prime(z,k)|^2\alpha_k(z)\right]\tilde{\gamma}(z,k)\ud z\right)^{\frac12}.
\end{equation}
 Moreover, for every $M>0$, we denote
\begin{equation}\label{eq:LambdaM}
\Lambda_{M,\tilde{\gamma}}:=\bigg\{\bar{\Phi}\in \mathcal{C}([0,T];{\bar{H}_{\tilde{\gamma}}}):
\sup_{t\in[0,T]}\|\bar{\Phi}(t)\|_{\bar{\mathbb{W}}^{1,2}_{\tilde{\gamma}}}^2
+\int_0^T\|\partial_t\bar{\Phi}(t,z,k)\|_{\bar{H}_{\tilde{\gamma}}}^2\ud t\le M\bigg\}.
\end{equation}

\subsection{SPDE on graph $\Gamma$}
Let $\mathcal{U}_0\subset L^\infty(D)$ be a separable Hilbert space and let
$\mathcal W=\{\mathcal W(t)\}_{t\ge 0}$ be an $\mathcal{U}_0$-cylindrical Wiener process defined on a complete filtered probability space $(\Omega,\mathscr{F},\{\mathscr{F}_t\}_{t\ge 0},\mathbb{P})$. As a consequence, $\mathcal W$ admits the following Karhunen--Loève expansion:
\begin{equation}\label{eq:WKL}
\mathcal W(t,x)=\sum_{i=1}^\infty e_i(x)\beta_i(t),\quad t \geq 0,x\in D,
\end{equation}
where $\{\beta_i\}_{i\in\mathbb N_+}$ is a sequence of independent Brownian motions and $\{e_i\}_{i=1}^\infty$ forms an orthonormal basis of the Hilbert space $\mathcal{U}_0$.
If $\mathcal U_0$ is finite dimensional, the summation in \eqref{eq:WKL} is a finite sum.
 In the sequel, we assume that \begin{equation}\label{eq:ei}
\sup_{x\in D}\sum_{i=1}^\infty|e_i(x)|^2<\infty.
\end{equation}
The process $\bar{\mathcal{W}}$ is the formal projection of $\mathcal{W}$ on the graph $\Gamma$, i.e., 
\begin{equation}\label{eq:barW}
\bar{\mathcal{W}}(t, z, k)=\sum_{j=1}^{\infty}e_j ^{\wedge}(z, k) \beta_j(t), \quad t \geq 0 ,\quad(z, k) \in \Gamma.
\end{equation}
Denote $\bar{\mathcal{U}}_0:=\{\varphi^\wedge\mid \varphi\in \mathcal{U}_0\}$ endowed with the norm $\|\varphi^\wedge\|_{\bar{\mathcal{U}}_0}:=\|\varphi\|_{\mathcal{U}_0}$. 
Then $\{e_i^\wedge\}_{i=1}^\infty$ forms an orthonormal basis of $\bar{\mathcal{U}}_0$, and $\bar{\mathcal{W}}$ is a $\bar{\mathcal{U}}_0$-cylindrical Wiener process. 

In this paper, we are mainly interested in the SPDE \eqref{eq:intro} on graph $\Gamma$ with small noise $(\text{i.e.}, \epsilon\in(0,1]$)
 and the initial value $\bar{u}^\epsilon(0)=\mathfrak{u}^{\wedge}$. Without further explanation, assume that $\mathfrak{u}\in H_\gamma$ and the functions
$b:\R\to\R$ and $g:\R\to\R$ are globally Lipschitz continuous. 
Next we consider the well-posedness of \eqref{eq:intro} in the space $H_\gamma$, where the graph weight function $\gamma:\Gamma\to (0,\infty)$ is always assumed to be bounded, continuous, and fulfills
 \begin{align}\label{eq:IkTk}
\sum_{k=1}^m \int_{I_k} \gamma(z, k) T_k(z) \ud z<\infty.
\end{align} Then
 $\gamma^\vee:\R^2\to(0,\infty)$ is also bounded, continuous and integrable.
From
\eqref{eq:ei}, \eqref{eq:wedge}, and \eqref{eq:IkTk}, we know that $\nu_\gamma$ defined by \eqref{eq:nugamma} is a finite non-negative measure on $(\Gamma,\mathscr{B}(\Gamma))$ and
\begin{equation}\label{eq:eiwedge}
\sup_{(z,k)\in\Gamma}\sum_{i=1}^\infty |e_i^\wedge(z,k)|^2<\infty.
\end{equation}
Moreover, the following contractivity properties hold (see e.g., \cite[section 3]{CF19})
\begin{gather}
\label{fvarphi}
\|\varphi^{\wedge}\|_{\bar{H}_\gamma}\le \|\varphi\|_{ H_\gamma },\quad \|f^{\vee}\|_{ H_\gamma }= \|f\|_{\bar{H}_\gamma}.
\end{gather}
Let $B$ and $G$ be the Nemytskii operators related to $b$ and $g$, respectively, i.e., for $f_1\in \bar{H}_\gamma$ and $f_2\in \bar{\mathcal{U}}_0$, 
\begin{align*}
B(f_1)(z,k)=b(f_1(z,k)),\quad G(f_1)(f_2)(z,k)=g(f_1(z,k))f_2(z,k),\quad (z,k)\in\Gamma.
\end{align*}
It follows from the Lipschitz continuity of $b:\R\to\R$ that $B:\bar{H}_\gamma\to\bar{H}_\gamma$ is globally Lipschitz continuous.
By the Lipschitz continuity of $g$, the inequalities \eqref{eq:eiwedge} and \eqref{eq:IkTk}, one has that the Nemytskii operator $G:\bar{H}_\gamma\to \mathcal{L}_2(\bar{\mathcal{U}}_0;\bar{H}_\gamma)$ is also globally Lipschitz continuous and of linear growth. Namely,
 for any $f_1,f_2\in \bar{H}_\gamma$,
\begin{gather}\label{eq:Glip}
\|B(f_1)-B(f_2)\|_{\bar{H}_\gamma}+\|G(f_1)-G(f_2)\|_{\mathcal{L}_2(\bar{\mathcal{U}}_0,\bar{H}_\gamma)}
\le C\|f_1-f_2\|_{\bar{H}_\gamma},\\\label{eq:Glingr}
\|B(f_1)\|_{\bar{H}_\gamma}+\|G(f_1)\|_{\mathcal{L}_2(\bar{\mathcal{U}}_0,\bar{H}_\gamma)}
\le C(1+\|f_1\|_{\bar{H}_\gamma}).
\end{gather} 
Denote by $\{\bar{S}(t):=e^{t\bar{L}}\}_{t\ge0}$ the Markov semigroup generated by $\bar{L}$. We impose the following assumption to ensure that
$\{\bar{S}(t)\}_{t\in[0,T]}$ is bounded in $L^q(\Gamma,\ud \nu_\gamma)$ for $q\ge 2$ (see Lemma \ref{lem:St} below).
\begin{assumption}\label{asp:gamma}
There exists a constant $C>0$ such that for any $(z,k)\in\Gamma$,
\begin{equation*}
\alpha_k(z)|\frac{\ud\gamma}{\ud z}(z,k)|^2\le CT_k(z)\gamma^2(z,k).
\end{equation*}
\end{assumption}
 \begin{lemma}\label{lem:St}
Under Assumption \ref{asp:gamma},
 for any $q\ge2$, there exists a constant $C:=C(q,T)>0$ such that for any $t\in[0,T]$,
\begin{equation}\label{eq:Stq}\|\bar{S}(t)\|_{\mathcal{L}(L^q(\Gamma,\ud \nu_\gamma))}\le C.
\end{equation}
\end{lemma}

The proof of Lemma \ref{lem:St} is provided in Appendix \ref{App:A1}. In particular, 
the special case of $q=2$ in Lemma \ref{lem:St} yields that \begin{equation}\label{Stbound}
 \|\bar{S}(t)\|_{\mathcal{L}(\bar{H}_\gamma )}
 \le C(T)\quad\forall~t\in(0,T].
\end{equation}
This, together with \cite[Theorem 7.5]{DZ14}, \eqref{eq:Glip}, and \eqref{eq:Glingr}, ensures 
 the well-posedness of \eqref{eq:intro}.
\begin{lemma}\label{lem:Wellposedgraph}
 Assume that $b$ and $g$ are globally Lipschitz continuous and $\mathfrak{u}\in H_\gamma$. Under Assumption \ref{asp:gamma}, there exists a unique mild solution $\bar{u}^\epsilon$ to \eqref{eq:intro}, i.e., for any $t\in[0,T]$,
\begin{equation*}
\bar{u}^\epsilon(t)=\bar{S}(t)\mathfrak{u}^\wedge+\int_0^t \bar{S}(t-s) B(\bar{u}^\epsilon(s)) \ud s+\sqrt{\epsilon}\int_0^t \bar{S}(t-s) G(\bar{u}^\epsilon(s)) \ud \bar{\mathcal{W}}(s).
\end{equation*}
Moreover, for any $p\ge1$, there exists some constant $C$ depending on $p$ such that
\begin{equation}\label{eq:udelta}\sup_{t\in[0,T]}\E\left[\|\bar{u}^\epsilon(t)\|_{\bar{H}_\gamma}^p\right]\le C(1+\|\mathfrak{u}\|_{H_\gamma}^p).
\end{equation}
\end{lemma}

\subsection{Connections between SPDE on graph and multiscale SPDE}
In this subsection, we illustrate that the SPDE \eqref{eq:intro} on graph $\Gamma$ can be viewed as the limiting equation of a multiscale SPDE on the domain $D$.

For a fixed $\epsilon>0$, we consider the multiscale SPDE \eqref{eq:intro2}
with the initial value $u_{\delta}^\epsilon(0)= \mathfrak{u}$, where $L_\delta$ is the infinitesimal generator of the diffusion process $X_\delta$.
With a slight abuse of notation,
we also let $B:H_\gamma\to H_\gamma$ and $G:H_\gamma\to \mathcal{L}_2(\mathcal{U}_0,H_\gamma)$ be the Nemytskii operators related to $b$ and $g$, respectively, i.e., for $\varphi_1\in H_\gamma$ and $\varphi_2\in\mathcal{U}_0$,
\begin{align*}
B(\varphi_1)(x)=b(\varphi_1(x)),\quad G(\varphi_1)(\varphi_2)(x)=g(\varphi_1(x))\varphi_2(x),\quad x\in\R^2.
\end{align*}
 For every $\delta>0$, we denote by $\{S_\delta(t):=e^{tL_\delta }\}_{t\ge0}$ the semigroup generated by $L_\delta $.
\begin{assumption}\label{Asp:SG-strong}
There exists a constant $K>0$ such that $\|S_\delta(t)\|_{\mathcal{L}( H_\gamma )}\le K$ for any fixed $\delta>0$ and $t\in[0,T]$.
\end{assumption}

The existence and uniqueness of the mild solution $u_\delta^\epsilon$ to \eqref{eq:intro2} follows from Assumption \ref{Asp:SG-strong} and the Lipschitz continuity of $b,g:\R\to \R$.
Since the constant $K$ in Assumption \ref{Asp:SG-strong} is independent of $\delta\in(0,1]$, it follows from 
 a stochastic factorization argument (see, e.g., \cite[section 5.3.1]{DZ14}) that, for any $p\ge1$, $u_\delta$ is uniformly bounded in $L^p(\Omega;\mathcal{C}([0,T]; H_\gamma ))$ with respect to $\delta\in(0,1]$ (cf. \cite[formula (5.8)]{CF19}).
Recall that as $\delta$ tends to $0$, the projected process $\Pi(X_\delta)$ converges to $\bar{Y}$, which corresponds to the infinitesimal generator $\bar{L}$. Furthermore, we assume
that
the asymptotics of $S_\delta(t)$ can be described by $\bar{S}(t)^\vee$ in the small $\delta$ limit.

\begin{assumption}\label{asp:S-S}
For any $\varphi \in H _{\gamma}$ and $0<\tau_0 < T$,
\begin{equation*}
\lim _{\delta \rightarrow 0} \sup _{t \in[\tau_0, T]}\left\|S_\delta(t) \varphi-\bar{S}(t)^{\vee} \varphi\right\|_{ H_\gamma }=0,
\end{equation*}
where $\bar{S}(t)^{\vee} \varphi:=(\bar{S}(t) \varphi^\wedge)^{\vee}$ for any $t\in[0,T]$.
\end{assumption}
\begin{remark}\label{rem:L2}
Combining Assumptions \ref{Asp:SG-strong} and \ref{asp:S-S}, for any $t\in(0,T]$, by choosing $\tau_0\in(0,t)$ and using the triangle inequality, it holds that for any $\delta>0$,
\begin{align*}
\|\bar{S}(t)^{\vee}\varphi\|_{ H_\gamma } \le K\|\varphi\|_{ H_\gamma } + \sup _{s\in[\tau_0, T]}\left\|S_\delta(s) \varphi-\bar{S}(s)^{\vee} \varphi\right\|_{ H_\gamma }.
\end{align*}
Letting $\delta\to 0^+$ and by the contractive property \eqref{fvarphi}, we can obtain \eqref{Stbound}.
We note that either Assumption \ref{asp:gamma} or the combination of Assumptions \ref{Asp:SG-strong} and \ref{asp:S-S} can be used to prove \eqref{Stbound}.
Consequently, Assumption \ref{asp:gamma} in Lemma \ref{lem:Wellposedgraph} can be replaced by Assumptions \ref{Asp:SG-strong} and \ref{asp:S-S} to ensure the well-posedness of \eqref{eq:intro}. 
\end{remark}

Under Assumption \ref{asp:S-S},
taking the limit as $\delta\to 0$ in \eqref{eq:intro2} yields \eqref{eq:intro}. In fact, following the proof of \cite[Theorem 7.2]{CF17} and \cite[Theorem 5.3]{CF19}, it can be verified that if Assumptions \ref{Asp:SG-strong} and \ref{asp:S-S} hold, then
 for every $\epsilon\in[0,1]$, $p\ge1$ and $0< \tau_0< T$,
 \begin{equation}\label{eq:u-u}
\lim_{\delta\to0}\E\bigg[\sup_{t\in[\tau_0, T]}\|u_\delta^\epsilon(t)-\bar{u}^\epsilon(t)^\vee\|_{H_\gamma}^p\bigg]=0.
\end{equation}
In particular, the asymptotical behavior described by \eqref{eq:u-u} holds for the settings discussed in sections \ref{S:ND} and \ref{S:FA}.

\section{Main results}\label{S:MR}
In this section, we introduce the mathematical framework for studying the moderate and large deviations of the SPDE on graph, motivated by the investigation of the asymptotical behavior of the models in \cite{CF17} and \cite{CF19} with small noise. In particular, the models from \cite{CF17} and \cite{CF19}
will be discussed in sections \ref{S:ND} and \ref{S:FA}, respectively.
 To illustrate our main results, we recall the definition of the LDP \cite{CX10,DZ10}. Let $\mathcal X$ be a Polish space. A real-valued function $\mathbf{I}:\mathcal X\rightarrow[0,\infty]$ is called a rate function if it is lower semi-continuous, i.e., for each $a\in[0,\infty)$, the level set $\mathbf{I}^{-1}([0,a]) := \{x \in \mathcal X \mid \mathbf{I}(x) \in [0,a] \}$ is a closed subset of $\mathcal X$. Furthermore, if the level set $\mathbf{I}^{-1}([0,a])$ is compact for any $a\in[0,\infty)$, then $\mathbf{I}$ is called a good rate function.

\begin{definition}
A family of $\mathcal X$-valued random variables $\{\mathbf{X}^\epsilon\}_{\epsilon>0}$ defined on a probability space $(\Omega,\mathscr F,\mathbb P)$ is said to satisfy an LDP on $\mathcal X$ with the speed $\upsilon(\epsilon)$ and the rate function $\mathbf{I}$ if for every Borel set $U$ of $\mathcal X$,
\begin{align*}
-\inf_{x\in U^\circ} \mathbf{I}(x)\le\liminf_{\epsilon\to 0}\frac{1}{\upsilon(\epsilon)}\ln\mathbb P\{\mathbf{X}^\epsilon\in U\}\le\limsup_{\epsilon\to 0}\frac{1}{\upsilon(\epsilon)}\ln\mathbb P\{\mathbf{X}^\epsilon\in U\}\leq-\inf_{x\in\bar U} \mathbf{I}(x),
\end{align*}
where $U^\circ$ and $\bar U$ denote the interior and closure of $U$ in $\mathcal X$, respectively.
\end{definition}

Passing to the limit as $\epsilon\to 0$ in \eqref{eq:intro} formally results in a PDE on graph
\begin{equation}\label{eq:utzk-deter}
\partial_t \bar{u}^0(t, z, k)=\bar{L} \bar{u}^0(t, z, k)+b(\bar{u}^0(t, z, k)).
\end{equation}
for $(t,z,k)\in[0, T]\times\Gamma$ with the initial value $\bar{u}^0(0)=\mathfrak{u}^{\wedge}$. Indeed, by \eqref{Stbound} as well as \eqref{eq:udelta}, it can be proved that for any $p\ge1$,
\begin{equation*}
\E\bigg[\sup_{t\in[0,T]}\|\bar{u}^\epsilon(t)-\bar{u}^0(t)\|_{\bar{H}_\gamma}^p\bigg]\le C(p,\|\mathfrak{u}\|_{H_\gamma},T)\epsilon^{\frac{p}{2}}.
\end{equation*}
In this work, we investigate deviations of $\bar{u}^\epsilon$ from the deterministic solution $\bar{u}^0$, as $\epsilon$ tends to $0$, namely, the large deviation estimate of the trajectory
\begin{equation}\label{eq:X}
\bar{X}^\epsilon(t,z,k):=\frac{\bar{u}^\epsilon(t,z,k)-\bar{u}^0(t,z,k)}{\sqrt{\epsilon}\lambda(\epsilon)},\quad t\in[0,T],
\end{equation}
where $\lambda(\epsilon)$ is some deviation scale influencing the asymptotic behavior of $\bar{X}^\epsilon$.
\begin{enumerate}
\item[(I)] The case $\lambda(\epsilon)=1/\sqrt{\epsilon}$ provides an LDP of $\bar{u}^\epsilon$. 
\item[(II)] If $\bar{X}^\epsilon$ satisfies an LDP with the deviation scale fulfilling
\begin{equation}\label{eq:MDPpara}\lambda(\epsilon)\to+\infty,\quad \sqrt{\epsilon}\lambda(\epsilon)\to0 \text{ as }\epsilon\to 0,
\end{equation}
then we say that $\bar{u}^\epsilon$ satisfies an MDP. This deviation property can be seen as an intermediate behavior between the central limit theorem and LDP.\end{enumerate}

For each $\phi\in L^2(0,T;\mathcal{U}_0)$, let $\bar{Z}^{\phi}$ be the unique mild solution to the following skeleton equation 
\begin{equation} \label{eq.skeleton-eq}
\left\{
\begin{split}
\partial_t\bar{Z}^{\phi}(t)&=\bar{L} \bar{Z}^{\phi}(t)
 +B(\bar{Z}^{\phi}(t))+G(\bar{Z}^{\phi}(t))\phi(t)^\wedge, \quad t\in[0,T],\\
 \bar{Z}^{\phi}(0) &= \mathfrak{u}^{\wedge}.
 \end{split}
 \right.
\end{equation}
Utilizing the Lipschitz continuity of $B:\bar{H}_\gamma\to\bar{H}_\gamma$ and $G:\bar{H}_\gamma\to \mathcal{L}_2(\bar{\mathcal{U}}_0;\bar{H}_\gamma)$ (see \eqref{eq:Glip}), as well as \eqref{Stbound}, it can be shown that for any $\phi\in L^2(0,T;\mathcal{U}_0)$,
there exists a unique mild solution $\bar{Z}^{\phi}$ to \eqref{eq.skeleton-eq} in $ \mathcal C([0, T] ; \bar{H}_\gamma)$ (see, e.g., \cite[Theorem 4.1]{CR04} for a similar proof). Namely, for any $t\in[0,T]$,
	\begin{equation}\label{eq:ut}
		\bar{Z}^{\phi}(t)=\bar{S}(t) \mathfrak{u}^{\wedge}+\int_0^t \bar{S}(t-s) B(\bar{Z}^{\phi}(s)) \ud s+\int_0^t \bar{S}(t-s) G(\bar{Z}^{\phi}(s)) \phi(s)^\wedge \ud s.
	\end{equation}
	According to \eqref{eq:ut}, \eqref{Stbound}, and the Cauchy--Schwarz inequality, it holds that 
	\begin{align*}
	\|\bar{Z}^{\phi}(t)\|^2_{\bar{H}_\gamma}
		&\le C\|\mathfrak{u}^{\wedge}\|^2_{\bar{H}_\gamma}+C(1+\|\phi\|_{L^2(0,T;\mathcal U_0)}^2)\int_0^t(1+\|\bar{Z}^{\phi}(s)\|_{\bar{H}_\gamma}^2)\ud s,\quad t\in[0,T],
	\end{align*}
	due to the linear growth of $B$ and $G$ in \eqref{eq:Glingr}.
	Taking supremum over $t\in[0,t_1]$ on both sides, we obtain that for any $t_1\in[0,T]$,
	\begin{align}\label{eq:Zphi}
		\|\bar{Z}^{\phi}\|^2_{\mathcal{C}([0,t_1];\bar{H}_\gamma)}\le C_1+C\left(1+\|\phi\|_{L^2(0,T;\mathcal U_0)}^2\right)\int_0^{t_1}\|\bar{Z}^{\phi}\|_{\mathcal{C}([0,s];\bar{H}_\gamma)}^2\ud s,
	\end{align}
	where $C_1$ is some constant dependent on $\| \mathfrak{u} \|_{ H_\gamma },\|\phi\|_{L^2(0,T;\mathcal U_0)}$, and $T$.
	This allows us to apply the Gronwall inequality to conclude 
	\begin{equation}\label{eq:Zvt}
		\|\bar{Z}^{\phi}\|_{\mathcal{C}([0,T];\bar{H}_\gamma)}\le C(\|\phi\|_{L^2(0,T;\mathcal{U}_0)},\| \mathfrak{u} \|_{ H_\gamma },T).
	\end{equation}	

To establish the LDP of \eqref{eq:intro}, we impose the following assumption.

\begin{assumption}\label{asp:compact}
There exists a bounded and continuous graph weight function $\gamma_1:\Gamma\to (0,\infty)$ such that 
\begin{enumerate}
\item[(i)] Assumption \ref{asp:gamma} and \eqref{eq:IkTk} hold with $\gamma$ replaced by $\gamma_1$;
\item[(ii)]
 For any $M>0$, the set $\Lambda_{M,\gamma_1}$ (see \eqref{eq:LambdaM})
 is a pre-compact set of $\mathcal {C}([0,T];\bar{H}_{\gamma})$.
 \end{enumerate}
\end{assumption}

\begin{remark}\label{rem:Hgamma}
By the Aubin--Lions lemma and \eqref{eq:LambdaM}, a sufficient condition for Assumption \ref{asp:compact}(ii) is that the space
$\bar{\mathbb{W}}^{1,2}_{\gamma_1}$
is compactly embedded into $\bar{H}_{\gamma}$.
\end{remark}

Notice that the graph weight function $\gamma_1$ can be chosen as $\gamma$ in Assumption \ref{asp:compact} if $\Lambda_{M,\gamma}$ is already a pre-compact set of $\mathcal {C}([0,T];\bar{H}_{\gamma})$.
Otherwise, one shall find a different graph weight function $\gamma_1$ 
to ensure Assumption \ref{asp:compact}(ii). 
Assumption \ref{asp:compact} will play an important role in verifying that \begin{equation}\label{eq:rate function-J0}
J_0(x):=\inf_{\left\{\phi\in L^2(0,T;\, \mathcal{U}_0 ),\,x=(\bar{Z}^\phi(\cdot))^\vee\right\}}\frac12 \int_0^T\|\phi(s)\|_{\mathcal{U}_0}^2\ud s,\quad x\in\mathcal{C}([0,T]; H_\gamma ).
\end{equation} is a good rate function.

Our first main result is the LDP of \eqref{eq:intro}.
\begin{theorem}\label{thm:LDP-graph}
 Assume that $b$ and $g$ are globally Lipschitz continuous and $\mathfrak{u}\in H_\gamma\cap H_{\gamma_1}$. 
Then under Assumptions \ref{asp:gamma} and \ref{asp:compact}, 
the family $\{(\bar{u}^\epsilon)^\vee\}_{\epsilon\in(0,1]}$ satisfies the LDP on $\mathcal{C}([0,T]; H_\gamma )$ as $\epsilon\to0$ with the speed $\epsilon^{-1}$ and the good rate function \eqref{eq:rate function-J0}.
\end{theorem}

Inspired by \cite{CH24}, we 
further 
extend the large deviation result (Theorem \ref{thm:LDP-graph}) to the multiscale SPDE in the continuous domain $D$.
Note that the solution $u_\delta^\epsilon$ to the multiscale SPDE \eqref{eq:intro2} 
depends on two parameters: the noise intensity $\epsilon$ and the multiscale parameter $\delta$. To study their joint limiting behavior, we relate them through $\delta=\psi(\epsilon)$ for some 
continuous function $\psi:(0,1]\to (0,\infty)$ satisfying 
\begin{equation}\label{eq:deltaepsilon}
\lim_{\epsilon\to 0}\psi(\epsilon)=0.
\end{equation}
Formally, the condition \eqref{eq:deltaepsilon} is imposed so that as $\epsilon\to 0$, the rescaled process $u^\epsilon_{\psi(\epsilon)}$ share the same limiting process $(\bar{u}^0)^\vee$ (see \eqref{eq:utzk-deter}) as the process $(\bar{u}^\epsilon)^\vee$.
Next we 
show that the convergence \eqref{eq:u-u} is consistent with respect to the small-noise limit, in the sense that $\{u^\epsilon_{\psi(\epsilon)}\}_{\epsilon\in(0,1]}$ satisfies an LDP with the same large deviation rate function (see \eqref{eq:rate function-J0}) as that of $(\bar{u}^\epsilon)^\vee$.

\begin{theorem}\label{thm:LDP}
Assume that $b$ and $g$ are globally Lipschitz continuous and $\mathfrak{u}\in H_\gamma\cap H_{\gamma_1}$. 
Then under Assumptions \ref{Asp:SG-strong}, \ref{asp:S-S}, and \ref{asp:compact}, for any fixed $\tau_0\in(0,T)$, $\{u^\epsilon_{\psi(\epsilon)}\}_{\epsilon\in(0,1]}$ satisfies the LDP on $\mathcal{C}([\tau_0,T]; H_\gamma )$ as $\epsilon\to0$ with the speed $\epsilon^{-1}$ and the good rate function
\begin{equation*}
J_{\tau_0}(x):=\inf_{\left\{\phi\in L^2(0,T; \mathcal{U}_0 ),\,x(t)=(\bar{Z}^\phi(t))^\vee, \,t\in[\tau_0,T]\right\}}\frac12 \int_0^T\|\phi(s)\|_{\mathcal{U}_0}^2\ud s,\quad x\in\mathcal{C}([\tau_0,T]; H_\gamma ).
\end{equation*}
\end{theorem}

Next, we impose \eqref{eq:MDPpara} and further investigate
the MDP of \eqref{eq:intro}.
If $b$ is continuously differentiable, 
then $B$ is also Fr\'echet differentiable on $\bar{H}_\gamma$. The Fr\'echet derivative ${\mathcal{D}B}:\bar{H}_\gamma\to \mathscr{L}(\bar{H}_\gamma)$ of $B$ is given by \begin{equation}\label{eq:B}
({\mathcal{D}B}(v_1)v_2)(z,k)=b^\prime(v_1(z,k))v_2(z,k),\quad v_1,v_2\in \bar{H}_\gamma.
\end{equation}
Formally, as $\epsilon$ tends to $0$, the process $\bar{X}^\epsilon$ {defined in \eqref{eq:X}} converges to the solution $\bar{X}^0$ to the following deterministic equation
 \begin{equation*}
	\partial_t \bar{X}^0(t)=\bar{L}\bar{X}^0(t)+{\mathcal{D}B}(\bar{u}^0(t))\bar{X}^0(t),\quad \bar{X}^0(0)=0.
\end{equation*}

For each $\phi\in L^2(0,T;\mathcal{U}_0)$, let $\bar{R}^{\phi}$ be the unique solution to the following skeleton equation
 \begin{equation} \label{eq:Nskeleton}
\left\{
\begin{split}
	\partial_t \bar{R}^\phi(t)&=\bar{L}\bar{R}^\phi(t)+{\mathcal{D}B}(\bar{u}^0(t))\bar{R}^\phi(t)+G(\bar{u}^0(t))\phi(t)^\wedge,\quad t\in(0,T],\\
	\bar{R}^\phi(0)&=0.
	\end{split}
	\right.
\end{equation}
The existence and uniqueness of the solution to \eqref{eq:Nskeleton} follow by arguments analogous to those for \eqref{eq.skeleton-eq} and are omitted for brevity. Our last main result is the MDP of \eqref{eq:intro}, which concerns the deviation of $\bar{X}^\epsilon$ from its typical value $\bar{X}^0$.
 \begin{theorem}\label{thm:MDP-graph}
Assume that $b$ is continuously differentiable with $b^\prime$ being $\alpha_0$th H\"older continuous with some $\alpha_0\in(0,1]$, and $g$ is globally Lipschitz continuous.
 Let $\mathfrak{u}\in L^{2\alpha_0+2}(\R^2,\gamma^\vee\ud x)\cap H_{\gamma_1}$,
 Assumptions \ref{asp:gamma} and \ref{asp:compact}, and \eqref{eq:MDPpara} hold. Then $\{(\bar{X}^\epsilon)^\vee\}_{\epsilon\in(0,1]}$ satisfies the LDP on $\mathcal{C}([0,T]; H_\gamma )$ as $\epsilon\to0$ with the speed $\lambda^2(\epsilon)$ and the good rate function
\begin{equation*}
J(x):=\inf_{\left\{\phi\in L^2(0,T; \mathcal{U}_0 ),\,x=(\bar{R}^\phi(\cdot))^\vee\right\}}\frac12 \int_0^T\|\phi(s)\|_{\mathcal{U}_0}^2\ud s,\quad x\in\mathcal{C}([0,T]; H_\gamma ).
\end{equation*}
\end{theorem}

Regarding the MDP of the rescaled process $\{u_{\psi(\epsilon)}^\epsilon\}_{\epsilon\in(0,1]}$, it is equivalent to studying the large deviation principle (LDP) for 
\begin{equation*}
U_\epsilon(t,x):=\frac{u_{\psi(\epsilon)}^\epsilon(t,x)-(\bar{u}^0)^\vee(t,x)}{\sqrt{\epsilon}\lambda(\epsilon)}=\frac{u_{\psi(\epsilon)}^\epsilon(t,x)-(\bar{u}^\epsilon)^\vee(t,x)}{\sqrt{\epsilon}\lambda(\epsilon)}+(\bar{X}^\epsilon)^\vee(t,x).
\end{equation*}
To determine the limit of $U_\epsilon$,
it is necessary to analyze the limit of $\frac{u_{\psi(\epsilon)}^\epsilon(t,x)-(\bar{u}^\epsilon)^\vee(t,x)}{\sqrt{\epsilon}\lambda(\epsilon)}.$
However, such an analysis would require establishing the strong convergence rate of $u_{\delta}^\epsilon$ to $\bar{u}^\epsilon$ with respect to $\delta$, which is beyond the scope of this work and is left for future investigation.
\section{Stochastic reaction diffusion equation on narrow domain}\label{S:ND}
In this section, we study the moderate and large deviation principles of the SPDE on graph introduced in \cite{CF17}. This model describes the asymptotics of stochastic reaction-diffusion equations in a bounded narrow domain $D_\delta:=\{(z,x_2)\in \R^2:(z,x_2/\delta)\in D\}$. 
Reaction-diffusion equations on narrow domains, with or without stochastic perturbations, arise in various contexts, such as models for the movement of molecular motors. In particular, one possible approach to modeling Brownian motors or ratchets is to conceptualize them as entities traversing a prescribed pathway. This pathway, along which the molecule or particle moves, can be interpreted as a tubular domain with multiple branches.

\subsection{Diffusion process on bounded domain}\label{S:NDX}
For each $\delta>0,$ consider the stochastic system \eqref{eq:X-ND} in Example \ref{Ex:1} with reflecting boundary conditions on 
a bounded open domain $D\subset\R^2$ with a smooth boundary $\partial D$.
Assume that $D$ fulfills the uniform exterior sphere condition (cf. \cite[Definition 1.1.6]{KKL01}) and the following hypothesis.
\begin{hypothesis}\label{hyp:3.1}
\begin{enumerate}
\item[$(D_1)$] There are only finitely many $z\in\R$ such that $\textup{\textbf{n}}_2(x)=0$ and $x=(z,x_2)\in \partial D$;
\item[$(D_2)$] For each $z\in \R$, the cross-section $C(z):=\{(z,x_2)\in D\}$ consists of a finite union of intervals. Namely, if $C(z)\neq \emptyset$, then there exists $N(z)\in \mathbb{N}$ and intervals $C_1(z), C_2(z),\ldots, C_{N(z)}(z)$ such that
$C(z)=\cup_{k=1}^{N(z)}C_k(z)$.
\item[$(D_3)$] If $z\in \R$ is such that $\textup{\textbf{n}}_2(z,x_2)\neq 0$, then $l_k(z):=|C_k(z)|>0$ for all $k=1,\ldots,N(z)$, where $|C_k(z)|$ denotes the length of $C_k(z)$.
\item[$(D_4)$] For each $z\in \R$ with $\textup{\textbf{n}}_2(z,x_2)= 0$, either $\textup{\textbf{n}}_1(z,x_2)>0$ for all $(z,x_2)\in D$ or $\textup{\textbf{n}}_1(z,x_2)<0$ for all $(z,x_2)\in D$.
\end{enumerate}
\end{hypothesis}
The slow motion $\mathcal{H}(X_\delta)$ of the process $X_\delta$ in \eqref{eq:X-ND} is its first component, i.e.,
\begin{equation}\label{eq:H-ND}
\mathcal{H}(x)=x_1,\quad x=(x_1,x_2)\in D.
\end{equation}
 In this setting, the infinitesimal generator of \eqref{eq:X-ND} is given by
 \begin{equation}\label{eq:L-deltaND}
 L_\delta =\frac12\textup{Tr}(\sigma_\delta\nabla^2)=\frac{1}{2}\frac{\partial^2}{\partial x_1^2}+\frac{1}{2\delta^2}\frac{\partial^2}{\partial x_2^2}
 \end{equation}
 equipped with the Neumann boundary conditions
$\nabla \varphi\cdot\sigma_\delta \textup{\textbf{n}}=0$ on $\partial D$ for $\varphi\in \textup{Dom}(L_\delta)$.

\begin{figure}[!htb]\label{fig:fish}
	\centering
	\includegraphics[width=0.5\linewidth]{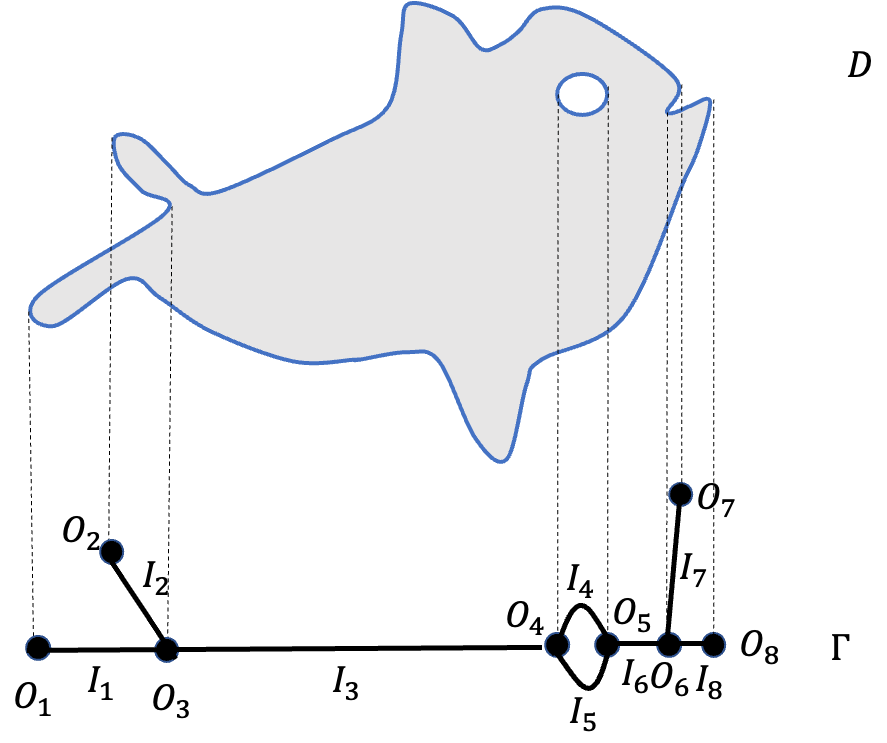}
	\caption{One example for the domain $D$ and the corresponding graph $\Gamma$ in section \ref{S:ND}.}\label{fig:fish}
\end{figure}

The graph $\Gamma$ is constructed by identifying the points of each connected component $C_k(z)$ of each cross-section $C(z)$. Each vertex corresponds to a connected component containing points $(z,x_2)\in \partial D$ with $\textup{\textbf{n}}_2(z,x_2)= 0$ (see Fig.\ \ref{fig:fish} for example).
It has been proven in \cite{FW12-PTRF} that for any initial value $x_0\in D$, the projected process $\Pi(X_\delta(\cdot))=(\mathcal{H}(X_\delta(\cdot)),k(X_\delta(\cdot)))$ with \eqref{eq:X-ND} and \eqref{eq:H-ND} converges in the sense of weak convergence of distributions in $\mathcal{C}([0,T];\Gamma)$ to a Markov process $\bar{Y}$ on $\Gamma$, corresponding to the infinitesimal generator \eqref{eq:barL} with
$\alpha_k(z)=T_k(z)=l_k(z).$
\subsection{Narrow domain asymptotics}
We take $\gamma(z,k)=1$ for all $(z,k)\in\Gamma$, under which
$ H_\gamma =L^2(D)$ and \begin{equation*}
\bar{H}_\gamma:=\left\{f: \Gamma \rightarrow \mathbb{R}: \sum_{k=1}^m \int_{I_k}|f(z, k)|^2 l_k(z) \ud z<\infty\right\}.
\end{equation*}
Concerning $\mathcal W$ in \eqref{eq:intro}, assume that $\{\mathfrak{e}_j\}_{j=1}^\infty$ is an orthonormal basis of $L^2(D)$ and 
\begin{equation}\label{eq:W-ND}
\mathcal W(t,x)=\sum_{j=1}^\infty Q\mathfrak{e}_j\beta_j(t),\quad t \geq 0,x\in D,
\end{equation}
where $Q\in\mathcal{L}(L^2(D))$ satisfies that $Q\mathfrak{e}_j=\mathfrak{q}_j\mathfrak{e}_j$ for a sequence $\{\mathfrak{q}_j\}_{j=1}^\infty\subset[0,\infty)$ with $\sum_{j=1}^\infty\mathfrak{q}_j^2<\infty$.
 Then $\mathcal{W}$ is a $\mathcal{U}_0$-cylindrical Wiener process with $\mathcal{U}_0:=Q(L^2(D))$ endowed with the norm $\|\cdot\|_{\mathcal{U}_0}:=\|Q^{-1}\cdot\|_{L^2(D)}$, where $Q^{-1}$ is the pseudo inverse of $Q$ in the case that $Q$ is not one-to-one.
The sequence $\{Q\mathfrak{e}_j;\mathfrak{q}_j>0\}$ forms an orthonormal basis of $\mathcal{U}_0$. To ensure \eqref{eq:ei} with $\{e_i,i\in \mathbb{N}^+\}=\{Q\mathfrak{e}_j;\mathfrak{q}_j>0\}$, we assume that
\begin{equation}\label{eq:qj}
\sup_{x\in D}\sum_{j:\mathfrak{q}_j>0}\mathfrak{q}_j^2|\mathfrak{e}_j(x)|^2<\infty.
\end{equation}
If $\{(\mathfrak{c}_j,\mathfrak{e}_j)\}_{j=1}^\infty$ forms an eigenpair of the negative {Neumann} Laplacian on $D$, i.e., $-\Delta \mathfrak{e}_j=\mathfrak{c}_j \mathfrak{e}_j$ with {$\nabla \mathfrak{e}_j\cdot\textup{\textbf{n}}=0$} on $\partial D$, then by the Sobolev embedding $W^{1+\epsilon_0,2}(D)\hookrightarrow L^\infty(D)$ for any $\epsilon_0>0$, a sufficient condition for \eqref{eq:qj} is
$\sum_{j:\mathfrak{q}_j>0}\mathfrak{q}_j^2\mathfrak{c}_j^{1+\epsilon_0}<\infty.
$ 
Furthermore, since
$\mathfrak{c}_j\sim j$ as $j\to\infty$, we can choose $\mathfrak{q}_j\sim j^{-(1+\epsilon_0)}$ to ensure the convergence of the series above.

We are now ready to verify Assumption \ref{asp:compact} with $\gamma_1=\gamma=1$.
\begin{lemma}\label{lem:compactND}
Let $\gamma\equiv1$, and $\alpha_k=T_k=l_k$ for $k=1,\ldots,m$. Then
 $\Lambda_{M,\gamma}$ is a pre-compact set of $\mathcal{C}([0,T];\bar{H}_{\gamma})$. 
\end{lemma}
\begin{proof}
By the chain rule and $\ud x=\frac{1}{|\nabla\mathcal{H}(x)|}\ud l_{z,k}\ud z$, we obtain
\begin{align}\label{eq:Dfvarphi}
\|\nabla f^\vee\|_{H_\gamma}^2
&=\sum_{k=1}^m\int_{I_k}\oint_{C_k(z)}|\nabla f(\mathcal{H}(x),k)|^2\frac{1}{|\nabla \mathcal{H}(x)|}\ud l_{z,k}\ud z\\\notag
&=\sum_{k=1}^m\int_{I_k}\oint_{C_k(z)}|f^\prime(\mathcal{H}(x),k)|^2\ud l_{z,k}\ud z\\\notag
&
=\sum_{k=1}^m\int_{I_k}|f^\prime(z,k)|^2 l_k(z)\ud z=
\| f^\prime\|_{\bar{H}_\gamma}^2.
\end{align}

Take a sequence $\{\bar\Phi_l\}_{l=1}^\infty\subset \Lambda_{M,\gamma}$, where $\Lambda_{M,\gamma}$ is defined in \eqref{eq:LambdaM} with $\alpha_k=T_k=l_k$ and $\gamma\equiv1$.
It follows from \eqref{fvarphi} and \eqref{eq:Dfvarphi} that for any $l\in\mathbb{N}^+$,
\begin{align*}
\sup_{t\in[0,T]}\|\bar\Phi_l(t)^\vee\|_{ H_\gamma }+
\sup_{t\in[0,T]}\|\nabla\bar\Phi_l(t)^\vee \|_{ H_\gamma }
+\int_0^T\|\partial_t\bar{\Phi}_l(t)^\vee\|^2_{ H_\gamma }\ud t\le M,
\end{align*}
where $H_\gamma =L^2(D)$.
By the Aubin--Lions lemma, the space
$\{u\in L^\infty(0,T;W^{1,2}(D))| \frac{\ud }{\ud t}u\in L^2(0,T;L^2(D))\}$ is compactly embedded into the space $\mathcal{C}([0,T];L^2(D))$, which implies that there is a subsequence $\{\bar\Phi_{l_i}\}_{i=1}^\infty$ of $\{\bar\Phi_l\}_{l=1}^\infty$ such that
$\{(\bar\Phi_{l_i})^\vee\}_{i=1}^\infty$ converges strongly in $\mathcal{C}([0,T];L^2(D))$, or equivalently $\{\bar\Phi_{l_i}\}_{i=1}^\infty$ converges strongly in $\mathcal{C}([0,T];\bar{H}_\gamma)$. This proves that $\Lambda_{M,\gamma}$ is a pre-compact set of $\mathcal{C}([0,T];\bar{H}_\gamma)$.
\end{proof}

\begin{corollary}Let $L_\delta$ and $\mathcal{H}$ be defined as in \eqref{eq:H-ND} and \eqref{eq:L-deltaND}, respectively, and assume that $\mathcal{W}$ is given by \eqref{eq:W-ND} with \eqref{eq:qj}. If $\gamma=\gamma_1\equiv1$, then Assumptions \ref{asp:gamma}, \ref{Asp:SG-strong}, \ref{asp:S-S}, and \ref{asp:compact} are fulfilled. 
Consequently, the conclusions of Theorems \ref{thm:LDP-graph}, \ref{thm:LDP}, and \ref{thm:MDP-graph} hold.
\end{corollary}
\begin{proof}
By the boundedness of $D$ and $T_k=l_k$, one has $\sum_{k=1}^m \int_{I_k} \gamma(z, k) T_k(z) \ud z=|D|<\infty$, and thus \eqref{eq:IkTk} holds.
 Notice that Assumption \ref{asp:gamma} holds naturally for $\gamma\equiv 1$.
According to \cite[section 5]{CF17}, for every $\delta>0$, $S_\delta(t)$ is a contraction on $H_\gamma$, which implies Assumption \ref{Asp:SG-strong} with $K=1$. Besides,
Assumptions \ref{asp:S-S} and \ref{asp:compact} follow from \cite[Corollary 5.3]{CF17} and Lemma \ref{lem:compactND}, respectively.
\end{proof}

 We note that under the setting of section \ref{S:ND}, the authors in \cite{CH24} have already established the LDP of the multiscale model \eqref{eq:intro2} with additive noise (i.e., $g=1$). They imposed a constraint $\delta=\epsilon^{\frac{1}{\kappa}}$ with $\kappa>0$ being arbitrarily fixed, which is consistent with the condition \eqref{eq:deltaepsilon} in Theorem \ref{thm:LDP}.

\section{Fast advection asymptotics for stochastic incompressible viscous fluids}\label{S:FA}
In this section, we show that the moderate and large deviation results apply to an SPDE on an unbounded graph \cite{CF19}, which models the fast advection asymptotics for stochastic incompressible viscous fluids.

\subsection{Diffusion process on unbounded graph}\label{S:FAX}
For every fixed $\delta>0$, consider the diffusion process \eqref{eq:FAXt} 
 on $\R^2$ in Example \ref{Ex:2},
where $\mathcal{H}:\R^2\to \R$ denotes the stream function describing the flow pattern of a fluid or gas in 2-dimensional space.
\begin{hypothesis}\label{hyp:H}
(i) $\mathcal{H}$ is fourth-order continuously differentiable with a bounded second derivative and
$\min _{x \in \mathbb{R}^2} \mathcal{H}(x)=0;\\
$
(ii) $\mathcal{H}$ has only a finite number of critical points $ \widehat{x}_1, \ldots, \widehat{x}_{m_1}$. 
The Hessian matrix $\nabla^2 \mathcal{H}({\widehat{x}_i})$ is non-degenerate for every $i=1, \ldots, m_1$, and $\mathcal{H}({ \widehat{x}_i}) \neq\mathcal{H}({ \widehat{x}_j})$ if $i \neq j$;\\
(iii) There exist $\mathfrak{a}_1, \mathfrak{a}_2, \mathfrak{a}_3>0$ such that
 for all $x \in \mathbb{R}^2$ with $|x|$ large enough,
\begin{equation*}
\mathcal{H}(x) \ge \mathfrak{a}_1|x|^2,\quad|\nabla \mathcal{H}(x)| \ge \mathfrak{a}_2|x|,\quad\Delta \mathcal{H}(x) \ge \mathfrak{a}_3.
\end{equation*}
\end{hypothesis}

For a fixed $\delta>0$, the infinitesimal generator $L_\delta $ of the diffusion process $X_\delta$ in \eqref{eq:FAXt} is
\begin{equation}\label{eq:Ldelta}L_\delta =\frac{1}{2} \Delta +\frac{1}{\delta}\nabla^{\perp} \mathcal{H}(x)\cdot \nabla.
\end{equation}

By identifying all points in $\mathbb{R}^2$ in the same connected component of a given level set $C(z):=\{x\in\R^2:\mathcal{H}(x)=z\}$ of the Hamiltonian $\mathcal{H}$, we obtain the graph $\Gamma$. Each vertex corresponds to a critical point of $\mathcal{H}$.
\begin{figure}[!htb]
\centering
\includegraphics[width=0.7\linewidth]{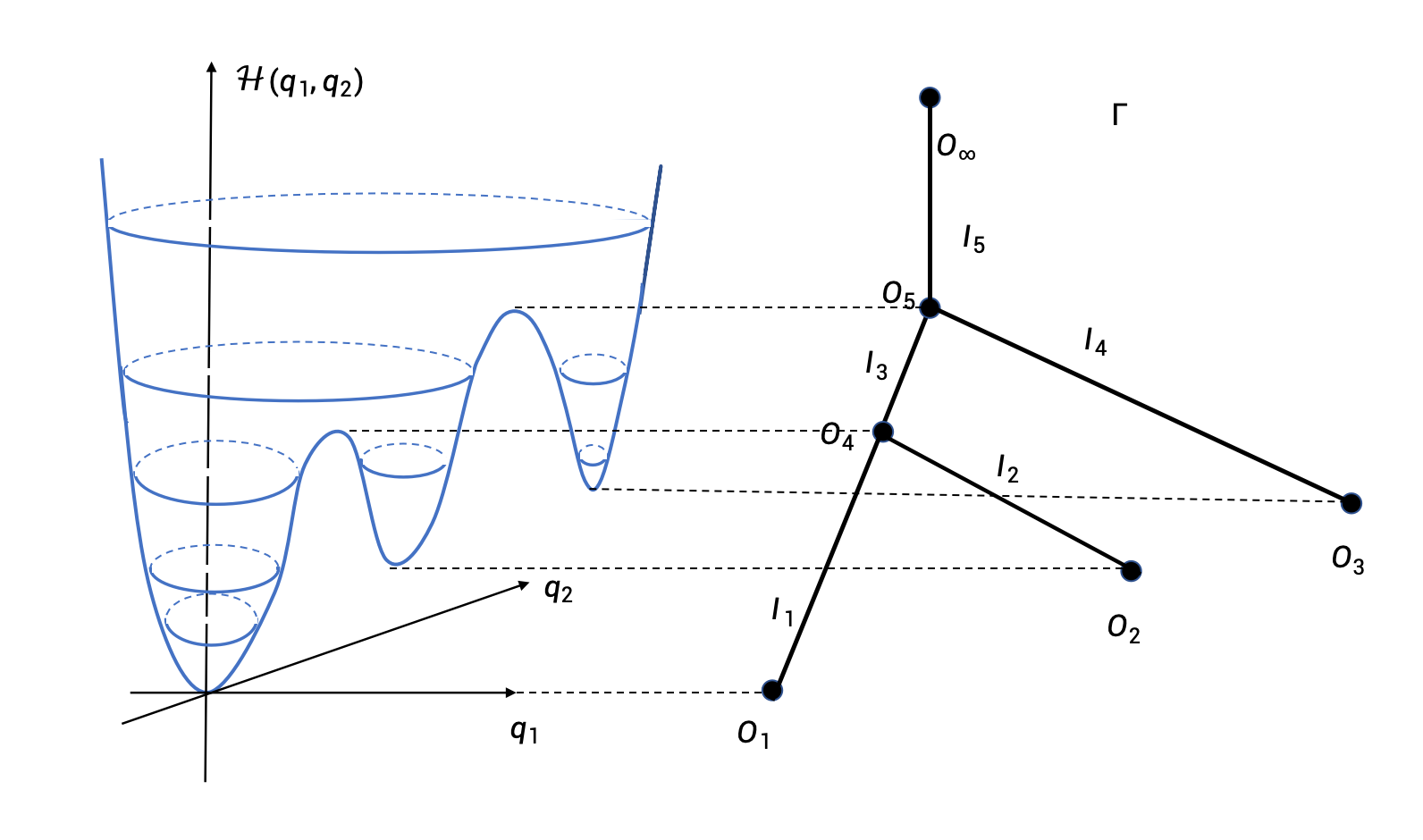}
\caption{One example for the Hamiltonian $\mathcal{H}$ with three local mimimizers ($O_1$, $O_2$, and $O_3$) and two saddle points ($O_4$ and $O_5$) and the graph $\Gamma$ in section \ref{S:FA}.}\label{fig:Hamilton}
\end{figure}
 We will also include $O_{\infty}$ among external vertices, the endpoint of the only unbounded interval on the graph, corresponding to the point at infinity (see Fig.\ \ref{fig:Hamilton}).
It has been shown in \cite[Chapter 8]{FW12} that for any initial value $x_0\in\R^2$,
 the projected process $\Pi(X_\delta(\cdot))$
 converges, in the sense of weak convergence of distributions in $\mathcal C([0,T];\Gamma)$, to a Markov process $\bar{Y}$ on $\Gamma$, corresponding to the infinitesimal generator \eqref{eq:barL}.

\subsection{Fast advection asymptotics}
Under {the setting of section \ref{S:FAX}}, the stochastic reaction diffusion advection equation \eqref{eq:intro2} 
describes how the density of particles moving along an incompressible flow in $D=\R^2$ changes over the time interval $[0,T/\delta]$, taking into account the effects of viscosity, the flow pattern, and the chemical reaction in which the particles are involved (see \cite{CF19,CS24}).

Assume that the stochastic process $\mathcal W$ in \eqref{eq:intro2} is a spatially homogeneous Wiener process
 with a finite non-negative definite symmetric spectral measure $\mu$. Let $\mathcal S(\R^2)$ be the set of smooth functions on $\R^2$ with rapid decay at infinity.
The operator $\mathscr{Q}:\mathcal S(\R^2)\times\mathcal S(\R^2)\to\R$, defined by
$$\mathscr{Q}(\varphi_1,\varphi_2):=\int_{\R^2}\Lambda(x)\int_{\R^2}\varphi_1(y)\varphi_2(y-x)\ud y\ud x,\quad\varphi_1,\varphi_2\in\mathcal S(\R^2),$$
is called the covariance form of $\mathcal W$, where $\Lambda$ is the Fourier transform of the spectral measure $\mu$.
By \cite[section 1.2]{PZ97}, the reproducing kernel space $\mathcal{U}_0$ of $\mathcal W$ can be identified with the dual of $\mathcal{S}_q $, where $\mathcal{S}_q $ is the completion of the set $\mathcal S(\R^2)/ker \mathscr{Q}$ with respect to the norm $q([\varphi]):=\sqrt{\mathscr{Q}(\varphi,\varphi)}$.
Denote by $L^2_{(s)}(\R^2,\ud \mu)$ the subspace of the Hilbert space $L^2(\R^2,\ud \mu;\mathbb{C})$ consisting of all functions $\varphi$ such that $ \overline{\varphi(-x)}=\varphi(x)$ for $x \in\R^2$.
 An orthonormal basis for the reproducing kernel space $\mathcal{U}_0$ is given by $\{e_j:=\widehat{\mathfrak{u}_j\mu}\}_{j\in\mathbb{N}^+}$ (cf. \cite[Proposition 1.2]{PZ97}), where $\{\mathfrak{u}_j\}_{j\in\mathbb{N}^+}$ is a complete orthonormal basis of the Hilbert space $L^2_{(s)}(\R^2,\ud \mu)$ and
$\widehat{\mathfrak{u}_j\mu}(x)=\int_{\R^2}e^{\textup{i}x\cdot\xi} \mathfrak{u}_j(\xi)\mu(\ud\xi)$, $x\in\R^2.$
Then the condition \eqref{eq:ei} follows from the finiteness of $\mu$. To be specific, for any $x\in \R^2$,
$\sum_{j=1}^\infty|\widehat{\mathfrak{u}_j\mu}(x)|^2=\mu(\R^2)<\infty$ (see \cite[formula (30)]{CS24}).

In contrast to section \ref{S:ND}, the graph weight function $\gamma_1$ is chosen here as $\gamma_1 := \sqrt{\gamma}$ rather than simply $\gamma$. We also note that one may choose $\gamma_1 = \gamma^{\iota}$ for any $\iota \in (0,1)$; for simplicity, we set $\iota = 1/2$ in this work. To verify Assumptions \ref{asp:gamma}, \ref{Asp:SG-strong}, \ref{asp:S-S} and \ref{asp:compact}, we suppose that 
 \begin{equation}\label{eq:dTdz}
\frac{\ud}{\ud z}T_k(z)\neq 0, \quad (z,k)\in\Gamma,
\end{equation}
and propose the following hypothesis.

\begin{hypothesis}\label{Asp:gamma-S}
Let $\gamma(z,k)=\vartheta(z)$ for every $(z,k)\in\Gamma$, where
$\vartheta:[0,\infty)\to(0,\infty)$ is a twice differentiable and bounded function such that
\begin{equation}\label{eq:sqrtgamma}
\sum_{k=1}^m \int_{I_k} \sqrt{\gamma(z,k)} T_k(z) \ud z<\infty.
\end{equation}
 Moreover, there exists some constant $C>0$ such that for any $z\ge 0$,
\begin{equation}\label{eq:zvartheta-strong}
z|\vartheta^{\prime\prime}(z)|+\sqrt{z}|\vartheta^{\prime}(z)|\le C\vartheta(z).
\end{equation}
\end{hypothesis}

\begin{remark}\label{Rmk:gamma}

(1) As stated in 
 \cite[pp.~10]{CS24}, under Hypothesis \ref{hyp:H}, there exists a constant $z_0>\max_{1\le i\le m_1} \mathcal{H}(\widehat{x}_i)$ such that $\{x\in \R^2,H(x)\ge z_0\}\subset \{x\in \R^2,\mathcal{H}(x)\ge\mathfrak{a}_1|x|^2\}$. 
If $\vartheta:[0,\infty)\to(0,\infty)$ is bounded and twice differentiable with 
\begin{equation*}
\vartheta(z)=c_0z^{-\kappa_1}\qquad \text{or} \qquad \vartheta(z)=c_0\exp(-\kappa_2(\sqrt{z}-\sqrt{z_0})),\quad \text{for } z>2z_0
\end{equation*}
for some $c_0>0$, $\kappa_1>2$ and $\kappa_2>0$, then Hypothesis \ref{Asp:gamma-S} holds.

 (2) The boundedness of $\vartheta$ and \eqref{eq:sqrtgamma} imply \eqref{eq:IkTk}. 
 Since $\vartheta$ is assumed to be twice differentiable and positive, the inequality \eqref{eq:zvartheta-strong} holds automatically for $z\in[0,\widetilde{z}_0]$ for any fixed $\widetilde{z}_0>0$. Furthermore,
one can obtain from \eqref{eq:zvartheta-strong} that for any $z\ge0$,
\begin{equation}\label{eq:gammaadmissible}
|z\vartheta^{\prime\prime}(z)|+|\vartheta^{\prime}(z)|\le C\vartheta(z),
\end{equation}
and thus \cite[Assumption 1]{CS24} holds. Therefore, the combination of \eqref{eq:dTdz}, \eqref{eq:gammaadmissible} and \eqref{eq:IkTk} indicates Assumptions \ref{Asp:SG-strong} and \ref{asp:S-S} (see \cite[Lemma 2.3]{CS24} and \cite[Corollary 5.3]{CF17}). Moreover, under Hypothesis \ref{Asp:gamma-S}, both \eqref{eq:gammaadmissible} and \eqref{eq:IkTk} continue to hold with $\vartheta$ replaced by $\sqrt{\vartheta}$. Therefore, Assumptions \ref{Asp:SG-strong} and \ref{asp:S-S} remain valid if $\gamma$ is replaced by $\sqrt{\gamma}$.

(3) {The condition \eqref{eq:dTdz} is used in verifying Assumption \ref{asp:S-S}. As pointed out in \cite{CH24}, this can be weakened such that \eqref{eq:dTdz} holds on the graph $\Gamma$ except for a finite number of points. Nevertheless, even this relaxed condition still excludes the typical example $H(x)=|x|^2$ for which the graph $\Gamma$ reduces to the half line $I_1:=[0,\infty)$ and the period $T_1(z)=\pi$.
}

\end{remark}

 We proceed to analyze the regularity and singularity of $\alpha_k$ and $T_k$ in \eqref{eq:AT}.
Without loss of generality, we assume that $I_m\cong(a_m,b_m)$ is the unique edge connecting $O_\infty$.
Recall that for each $k=1,\ldots,m$, we identify $I_k$ with $(a_k,b_k)\subset \R$, where $0\le a_k<b_k<\infty$ for $k=1,\ldots,m-1$, $a_m=\max_{k=1}^{m-1}b_k$ and $b_m=\infty$. 
By \cite[Lemma 1.1]{FW12}, the functions $T_k$ and $\alpha_k$ are continuously differentiable in $(a_k,b_k)$. Hence, for any $[\alpha^\prime,\beta^\prime]\subset (a_k,b_k)$, there exist two positive constants $c$ and $C$ depending on $\alpha^\prime$ and $\beta^\prime$ such that for any $z\in[\alpha^\prime,\beta^\prime]$,
$$c\le \min\{\alpha_k(z),T_k(z)\}\le \max\{\alpha_k(z),T_k(z)\}\le C.$$
In contrast, $\alpha_k$ and $T_k$ possibly exhibit {singularities or degeneracy} near the vertices.
Indeed, if $(z,k)$ approaches an endpoint of an edge $I_k$, corresponding to a vertex $O_i=(\mathcal{H}(x_i),k)$, then (see \cite[pp. 266-267]{FW12} and \cite[pp. 501]{CF19})
\begin{align}\label{eq:A}
&\lim_{(z,k)\to O_i}\alpha_k(z)\sim\begin{cases} \textup{const}\cdot(z- \mathcal{H}(x_i)),\quad &\text{if $x_i$ is a local minimizer or maximizer},\\
\textup{const},\quad& \text{if $x_i$ is a saddle point},\\
\textup{const}\cdot z,\quad &\text{if $O_i=O_\infty$};
\end{cases}\\\label{eq:T}
&\lim_{(z,k)\to O_i}T_k(z)\sim\begin{cases}\textup{const},&\text{if $x_i$ is a local minimizer or maximizer},\\
\textup{const}\cdot|\log|z-\mathcal{H}(x_i)||,& \text{if $x_i$ is a saddle point},\\
\textup{const},& \text{if $O_i=O_\infty$}.
\end{cases}
\end{align}
Here \textup{const} denotes some positive constant depending on $k$ and $O_i$, and may differ from line to line. 
To verify Assumption \ref{asp:compact}(ii), we make full use of the following embedding theorems.

\begin{lemma}\cite[Theorem 4]{GU09}\label{lem:GU09-4} Let $D^\prime\subset \R^n$ be an embedding domain, $1\le s< r<nq/(n-q)$, $q\le p$, $1<p<\infty$ and
$
\int_{D^\prime}|w(x)|^{\frac{r}{r-s}}\ud x<\infty.
$
Then the embedding operator $i:W^{1,p}(D^\prime)\hookrightarrow L^s(D^\prime;w(x)\ud x)$ is compact.
\end{lemma}
We refer readers to \cite[Definition 1]{GU09} for the precise definition of an embedding domain, and note that bounded Lipschitz domains $D^\prime\subset\R^n$ are examples of embedding domains. 
\begin{lemma}\cite[Theorem 2.4]{CRW13}\label{lem:CRW1324} Suppose that $1\le s<\infty$ and $\Omega$ is an $s$-John domain in $\R^n$. Let $p,a,b$ satisfy $1\le p<\infty$, $a\ge0$, $b\in\R$, and $b-a<p$. If
$n+a>s(n-1+b)-p+1,$
then, for any $1\le q<\infty$ such that
$$\frac1q>\max\left\{\frac{1}{p}-\frac{1}{n},\frac{s(n-1+b)-p+1}{(n+a)p}\right\},$$
 the space
$L^1_{\rho^a}(\Omega)\cap E^p_{\rho^b}(\Omega):=\{F\in L^1(\Omega,\rho(x,\Omega^c)^a\ud x)\mid \nabla F\in L^p(\Omega,\rho(x,\Omega^c)^b\ud x)\}$ is compactly embedded into $L^q(\Omega,\rho(x,\Omega^c)^a\ud x)$.
Here, $\rho(x,\Omega^c)$ denotes the Eucludean distance between the sets $\{x\}$ and $\Omega^c:=\R^n-\Omega$, and
the space
$L^1_{\rho^a}(\Omega)\cap E^p_{\rho^b}(\Omega)$ is endowed with the norm
$$\|F\|_{L^1_{\rho^a}(\Omega)\cap E^p_{\rho^b}(\Omega)}:=\|F\|_{L^1(\Omega,\rho(x,\Omega^c)^a\ud x)}+\|\nabla F\|_{L^p(\Omega,\rho(x,\Omega^c)^b\ud x)},~ F\in L^1_{\rho^a}(\Omega)\cap E^p_{\rho^b}(\Omega).$$
\end{lemma}
For the concept of $s$-John domain, we refer to \cite[pp. 29]{CRW13} or \cite[Definition 6.2.1]{HH19} for the precise definition. In particular, finite intervals are $1$-John domains in $\R$ (see, e.g., \cite[Example 6.2.2]{HH19}). 
Now we are in a position to show that Assumption \ref{asp:compact}(ii) holds with $\gamma_1=\sqrt\gamma$.

\begin{lemma}\label{lem:compactFA}
Let the functions $\alpha_k$ and $T_k$ be defined by \eqref{eq:AT} with $\mathcal{H}$ satisfying Hypothesis \ref{hyp:H}.
If $\gamma$ satisfies Hypothesis \ref{Asp:gamma-S} with
\begin{equation}\label{eq:vartheta}\lim_{R\to \infty}\sup_{z>R}\gamma(z,m)=0,
\end{equation}
then
 $\Lambda_{M,\sqrt\gamma}$ is a pre-compact set of $\mathcal{C}([0,T];\bar{H}_{\gamma})$.
\end{lemma}

\begin{proof}
In view of Remark \ref{rem:Hgamma}, it suffices to show that the space $\bar{\mathbb{W}}^{1,2}_{\sqrt{\gamma}}$ defined in \eqref{eq:Lambda} with $\tilde{\gamma}=\sqrt{\gamma}$ is compactly embedded into $\bar{H}_\gamma$.
 To handle the external and internal vertices separately, we partition edges $I_k$ as follows.
 For $\ell\in\{1,2,\ldots,2m\}$, let $\widetilde{T}_\ell:=T_{\ell\pmod m}$, $\widetilde{\alpha}_\ell:=\alpha_{\ell \pmod m}$, and
\begin{align*}
\widetilde{I}_{\ell}=\begin{cases}(a_\ell,\frac12(a_\ell+b_\ell)),\quad &\ell=1,\ldots,m-1,\\
(a_m,a_m+1),\quad &\ell=m,\\
(\frac12(a_{\ell-m}+b_{\ell-m}),b_{\ell-m}),\quad &\ell=m+1,\ldots,2m-1,\\
(a_m+1,\infty),\quad &\ell=2m. 
 \end{cases}
 \end{align*}
 Consequently, $\bar{H}_{\sqrt\gamma}=\bigoplus_{k=1}^{m} L^2(I_k,T_k{\sqrt\gamma}\ud z)=\bigoplus_{\ell=1}^{2m} L^2(\widetilde{I}_{\ell},\widetilde{T}_\ell{\sqrt\gamma}\ud z)$. 
 For each $\ell=1,2,\ldots,2m$, we introduce a subspace $\mathbb{G}_{\ell,\sqrt{\gamma}}$ of $L^2(\widetilde{I}_{\ell},\widetilde{T}_\ell{\sqrt\gamma}\ud z)$ defined by
 \begin{equation*}
 \mathbb{G}_{\ell,\sqrt{\gamma}}:=\left\{w:\widetilde{I}_{\ell}\to \R\mid \|w\|_{\mathbb{G}_{\ell,\sqrt{\gamma}}}^2:=\int_{\widetilde{I}_{\ell}}|w|^2\widetilde{T}_\ell{\sqrt\gamma}\ud z+\int_{\widetilde{I}_{\ell}}|w^\prime|^2\widetilde{\alpha}_\ell{\sqrt\gamma}\ud z<\infty\right\}.
 \end{equation*}
Then by \eqref{eq:Lambda}, the direct sum space $\bigoplus_{\ell=1}^{2m}\mathbb{G}_{\ell,\sqrt{\gamma}}$ coincides with $\bar{\mathbb{W}}^{1,2}_{\sqrt{\gamma}}$.
By a diagonal argument, the compact embedding of $\bar{\mathbb{W}}^{1,2}_{\sqrt{\gamma}}$ into $\bar{H}_{\gamma}$ will follow from \emph{Claims 1} and \emph{2}. 

\emph{Claim 1.
 For each $\ell=1,2,\ldots,2m-1$,
 $\mathbb{G}_{\ell,\sqrt{\gamma}}$
 is compactly embedded into $L^2(\widetilde{I}_\ell,\widetilde{T}_\ell\gamma\ud z)$.}

 According to Hypothesis \ref{Asp:gamma-S}, $\gamma$ is uniformly bounded from below and above by some positive constants $c$ and $C$, respectively, on $\cup_{\ell=1}^{2m-1}I_\ell$. Namely, 
$c\le \gamma(z,k)\le C$
for all $(z,k)\in\cup_{\ell=1}^{2m-1}I_\ell$.
 For each $\ell=1,\ldots,2m-1$, the edge $\widetilde{I}_\ell$ connects only one vertex of $\Gamma$, which corresponds to a critical point of $\mathcal{H}$.
We will only prove \emph{Claim 1} for the cases $1\le \ell\le m$ where $\widetilde{I}_\ell=(a_\ell,d_\ell)$ with $d_\ell:=\frac12(a_\ell+b_\ell)$, since the proof for the cases $m+1\le \ell\le 2m-1$ is analogous.

 \emph{Case 1. $a_\ell$ corresponds to a saddle point of $\mathcal{H}$.}

 Since $\widetilde{T}_\ell(z)\sim \textup{const}\cdot|\log(z-a_\ell)|$ for $z\to a_\ell^+$, we have
\begin{align*}
\int_{\widetilde{I}_\ell}|\widetilde{T}_\ell(z)\gamma(z,\ell)|^2\ud z=\int_{a_\ell}^{d_\ell}|T_\ell(z)\gamma(z,\ell)|^2\ud z<\infty.
\end{align*}
Furthermore, by applying Lemma \ref{lem:GU09-4} with $w=\widetilde{T}_\ell\gamma$, $D^\prime=\widetilde{I}_\ell$, $n=q=1$, $p=2$, $s=2$, and $r=4$, we conclude that
$W^{1,2}(\widetilde{I}_\ell)$ is compactly embedded into $L^2(\widetilde{I}_\ell,\widetilde{T}_\ell\gamma\ud z)$. Therefore,
$\mathbb{G}_{\ell,\sqrt{\gamma}}\subset W^{1,2}(\widetilde{I}_\ell)$ is also compactly embedded into $L^2(\widetilde{I}_\ell,\widetilde{T}_\ell\gamma\ud z)$.

\emph{Case 2. $a_\ell$ corresponds to a local minimizer or maximizer of $\mathcal{H}$.}

We apply Lemma \ref{lem:CRW1324} with the bounded domain $\Omega=\widetilde{I}_\ell=(a_\ell,d_\ell)\subset \R$, taking $s=1$, $a=0$, $b=1$, $p=q=2$, and $n=1$. Note that $n+a=1$ and $s(n-1+b)-p+1=0$. This yields that $L^1(\widetilde{I}_\ell)\cap E_{\rho\ud z}^2(\widetilde{I}_\ell)$ is compactly embedded into $L^2(\widetilde{I}_\ell)$, where
$$E_{\rho\ud z}^2(\widetilde{I}_\ell):=\left\{w:\widetilde{I}_\ell\to\R:\int_{\widetilde{I}_\ell}|w^\prime(z)|^2\min\{z-a_\ell,d_\ell-z\}\ud z<\infty\right\}.$$
Since $\gamma$ and $\widetilde{T}_\ell$ are uniformly bounded from below and above on $\widetilde{I}_\ell$, it follows that
$L^1(\widetilde{I}_\ell,\widetilde{T}_\ell\gamma\ud z)\cap E_{\rho\ud z}^2(\widetilde{I}_\ell)$ is compactly embedded into $L^2(\widetilde{I}_\ell,\widetilde{T}_\ell\gamma\ud z)$. { Referring to the asymptotic behaviors of $\widetilde{\alpha}_\ell$ and $\widetilde{T}_\ell$ near the local minimizer or maximizer of $\mathcal{H}$ (see \eqref{eq:A} and \eqref{eq:T}), there exists a positive constant $c$ such that $\min\{ z-a_\ell,d_\ell-z\}\le c\alpha_\ell(z)$ for all $z\in \widetilde{I}_\ell$. This observation, combined with the H\"older inequality, implies that
}
 $\mathbb{G}_{\ell,\sqrt{\gamma}}\subset L^1(\widetilde{I}_\ell,\widetilde{T}_\ell\gamma\ud z)\cap E_{\rho\ud z}^2(\widetilde{I}_\ell).$ As a consequence, one has that $\mathbb{G}_{\ell,\sqrt\gamma}$ is compactly embedded into $L^2(\widetilde{I}_\ell,\widetilde{T}_\ell\gamma\ud z)$.

\emph{Claim 2. Under \eqref{eq:vartheta}, $\mathbb{G}_{2m,\sqrt{\gamma}}$ is compactly embedded into $L^2(\widetilde{I}_{2m},\widetilde{T}_{2m}\gamma\ud z)$.}

Notice that
$L^2(\widetilde{I}_{2m},\widetilde{T}_{2m}\gamma\ud z)=L^2((a_m+1,\infty),T_m\gamma\ud z)$.
Let $\{w_n\}_{n=1}^\infty$ be a sequence of functions in $\mathbb{G}_{2m,\sqrt\gamma}$ with $w_n\rightharpoonup 0$ (weak convergence) as $n\to \infty$ and $\{\|w_n\|_{\mathbb{G}_{2m,\sqrt\gamma}}\}_{n=1}^\infty$ being uniformly bounded.
For every $R>a_{m}+1$,
\begin{align}\label{eq:wn}
& \int_{a_m+1}^\infty|w_n(z)|^2T_m(z)\gamma(z,m)\ud z\\\notag
&\le \int_{a_m+1}^R|w_n(z)|^2T_m(z)\gamma(z,m)\ud z+\sup_{n\ge1}\|w_n\|^2_{\mathbb{G}_{2m,\sqrt{\gamma}}}\sup_{z\ge R}\sqrt{\gamma(z,m)},
\end{align}
where the second term on the right-hand side tends to $0$ as $R\to\infty$, due to \eqref{eq:vartheta}.
For any fixed $R>0$,
 the function $\alpha_m$ and $T_m$ are bounded from below on $[a_m+1,R]$ by some positive constant depending on $R$ and $a_m$. By the Rellich--Kondrachov theorem (see, e.g., \cite[section 5.7]{EL10}), $W^{1,2}(a_m+1,R)$ 
 is compactly embedded into $L^2(a_m+1,R)$. This means that the inclusion $i:W^{1,2}(a_m+1,R)\to L^2(a_m+1,R)$ is a completely continuous operator, mapping weak convergence to strong convergence.
Since $w_n\rightharpoonup 0$ as $n\to\infty$ and $\{w_n|_{[a_m+1,R]}\}_{n=1}^\infty$ is a bounded sequence in $W^{1,2}(a_m+1,R)$, it follows that for any fixed $R>a_m+1$,
$$\lim_{n\to \infty}\int_{a_m+1}^R|w_n(z)|^2T_m(z)\gamma(z,m)\ud z=0.$$
Passing to the limits as $n\to\infty$ and subsequently as $R\to \infty$ on both sides of \eqref{eq:wn} yields that $w_n\to0$ in $L^2((a_m+1,\infty),T_m\gamma\ud z)$. This established that the inclusion $i:\mathbb{G}_{2m,\sqrt\gamma}\to L^2((a_m+1,\infty),T_m\gamma\ud z)$ is a completely continuous operator. Together with the reflexivity of the Hilbert space $\mathbb{G}_{2m,\sqrt\gamma}$, this confirms the compactness of the inclusion $i:\mathbb{G}_{2m,\sqrt\gamma}\to L^2((a_m+1,\infty),T_m\gamma\ud z)$. The proof is thus completed.
\end{proof}

\begin{corollary}
Assume that $\mathcal W$ is a spatially homogeneous Wiener process
 with a finite non-negative definite symmetric spectral measure $\mu$.
 Let $L_\delta$ be defined as in \eqref{eq:Ldelta} with the function $\mathcal{H}:\R^2\to \R$ satisfying
 Hypothesis \ref{hyp:H} and \eqref{eq:dTdz}.
If $\gamma_1=\sqrt{\gamma}$ and $\gamma$ satisfy Hypothesis \ref{Asp:gamma-S} and \eqref{eq:vartheta}, 
 then Assumptions \ref{asp:gamma}, \ref{Asp:SG-strong}, \ref{asp:S-S}, and \ref{asp:compact} are valid. Therefore, the conclusions of Theorems \ref{thm:LDP-graph}, \ref{thm:LDP}, and \ref{thm:MDP-graph} hold.
\end{corollary}
\begin{proof}
The verification of Assumptions \ref{Asp:SG-strong} and \ref{asp:S-S}, as well as \eqref{eq:IkTk}, has been illustrated in Remark \ref{Rmk:gamma}(2). Combining \eqref{eq:A}, \eqref{eq:T} with \eqref{eq:gammaadmissible}, a sufficient condition for both Assumption \ref{asp:gamma} and {Assumption \ref{asp:compact}(i)} 
is 
$\limsup_{z\to\infty}\sqrt{z}|\vartheta^\prime(z)|/ \vartheta(z)< \infty$, which holds automatically due to \eqref{eq:zvartheta-strong} in Hypothesis \ref{Asp:gamma-S}. Finally, Assumption \ref{asp:compact}(ii) holds due to Lemma \ref{lem:compactFA}.
\end{proof}

\subsection{Counterexample to compact Sobolev embedding}
In this subsection, we show the existence of $\gamma$ for which the embedding $\bar{\mathbb{W}}_{\gamma}^{1,2}\hookrightarrow \bar{H}_\gamma$ is not compact, whereas $\bar{\mathbb{W}}_{\sqrt\gamma}^{1,2}\hookrightarrow \bar{H}_\gamma$ is. 
This illustrates the necessity of introducing another graph weight function $\gamma_1$ in Assumption \ref{asp:compact}.

\begin{proposition}\label{lem:counterexample}
Assume that Hypothesis \ref{Asp:gamma-S} holds with 
$\vartheta(z)=c_0z^{-\kappa_1}$ for some $\kappa_1>1$ and all $z>2z_0$, where $z_0$ is the same as in Remark \ref{Rmk:gamma}(1). Then the embedding $\bar{\mathbb{W}}^{1,2}_{\gamma}\hookrightarrow \bar{H}_\gamma$ is not compact.
\end{proposition}

The proof of Proposition \ref{lem:counterexample} relies on the following result.
\begin{lemma}\label{lem:non-compact}
Let $h:(R,\infty)\to [0,\infty)$ be a measurable function for some $R>0$. Define
 $\mathbb{F}_R^h$ as the space of functions 
 $\xi\in L^{2}((R,\infty),h\ud z)$ with $\xi^\prime\in L^{2}((R,\infty),\!zh(z)\ud z)$, equipped with the norm
$\|\xi\|_{\mathbb{F}_R^h}:=(\|\xi\|^2_{L^{2}((R,\infty),h\ud z)}+\|\xi^\prime\|_{L^{2}((R,\infty),zh(z)\ud z)}^2)^{\frac12}.$
If there exist some $\varepsilon,\varrho>0$ and a sequence $\{r_n\}_{n\in\mathbb{N}^+}\subset (0,\infty)$ with $\lim_{n\to\infty}r_n=\infty$ such that for any $n\in\mathbb{N}^+$,
\begin{equation}\label{eq:com}
\int_{r_n}^\infty h(z)\ud z> \varrho\int_{r_n-\varepsilon}^{r_n} z h(z)\ud z,
\end{equation}
then
the embedding
$
\mathbb{F}_R^h \hookrightarrow L^{2}((R,\infty);h\ud z)
$
is not compact.
\end{lemma}
\begin{proof}
Denote $ \mathfrak{p}_0(\ud z)=h(z)\ud z$, $ \mathfrak{p}_1(\ud z )=zh(z)\ud z$ and $Q_r=(r,\infty)$ for $r>1$.
Then the condition \eqref{eq:com} can be represented as
\begin{equation}\label{eq:contrary}
\mathfrak{p}_0 ( {Q}_{r_n}) > \varrho\mathfrak{p}_1\{[r_n - \varepsilon,r_n]\}. 
\end{equation}
For each integer $n\ge1$, let us choose a smooth function ${\varphi }_{n}: [0,\infty)\to[0,1]$ such that
${\varphi }_{n} \equiv 1$ on ${Q}_{r_n},{\varphi }_{n} \equiv 0$ on $[0,r_n - \varepsilon]$, and
$\mathop{\sup }\limits_{{n\in {\mathbb{N}}^{ + }}}\left| {\varphi^\prime_{n}}\right| \leq M_\varepsilon < \infty.$
By introducing ${\xi }_{n} := K_n{\varphi }_{n}$ with $K_n = 1/\sqrt{ \mathfrak{p}_0 ( {Q}_{r_n})}$, it follows from \eqref{eq:contrary} that for any $n\in\mathbb{N}^+$ with $r_n-\varepsilon\ge1$,
\begin{align*}
\int_{0}^\infty |{\xi }_{n}^\prime(z)|^2 zh(z)\ud z&=\int_{r_n-\varepsilon}^{r_n}|{\xi }_{n}^\prime(z)|^2 zh(z)\ud z\le K_n^2M_\varepsilon^2\mathfrak{p}_1\{[r_n-\varepsilon,r_n]\}\le M_\varepsilon^2\varrho^{-1},\\
\int_{0}^\infty |{\xi }_{n}(z)|^2 h(z)\ud z&\le K_n^2\mathfrak{p}_0 (Q_{r_n-\varepsilon}) \le K_n^2\left(\mathfrak{p}_0 ( {Q}_{r_n})+\mathfrak{p}_1\{[r_n-\varepsilon,r_n]\}\right)\le 1+\varrho^{-1}.
\end{align*}
Hence $\left\{ {\xi }_{n}\right\}_{n\in\mathbb{N}^+}$ is a bounded
sequence in $\mathbb{F}_R^h$. 
If a subsequence of $\left\{{\xi }_{n}\right\}_{n\in\mathbb{N}^+}$ is convergent, it can only converge to
$\xi \equiv 0$, which gives rise to a contradiction with the fact that $ K_n^{2}{\int }_{{Q}_{r_n}}{\left| {\varphi }_{n}\right| }^{2}h(z)\ud z = 1$
for all $n\in\mathbb{N}^+$. The proof is completed.
\end{proof}

\emph{Proof of Proposition \ref{lem:counterexample}.}
Without loss of generality, we assume that $c_0=1$ for simplicity. We choose $\varepsilon>0$ and $\varrho>0$ such that $(\kappa_1-1)\varepsilon \varrho<\frac12$. This ensures the existence of a constant $R_1>2z_0$ for which
$
(\frac{r}{r-\varepsilon})^{1-\kappa_1}\ge\frac12> (\kappa_1-1)\varepsilon \varrho
$
for any $r>R_1$.
Consequently, for any $r>R_1$, we have
$$\varrho\int_{r-\varepsilon}^r z \vartheta(z)\ud z=\varrho\int_{r-\varepsilon}^r z^{1-\kappa_1}\ud z\le (r-\varepsilon)^{1-\kappa_1}\varepsilon \varrho< \frac{r^{1-\kappa_1}}{\kappa_1-1}=\int_r^\infty \vartheta(z)\ud z.$$
Then Lemma \ref{lem:non-compact} yields that for $R>2z_0$,
 the embedding $\mathbb{F}_{R}^\vartheta\hookrightarrow L^{2}((R,\infty);\vartheta\ud z)$ is not compact. Specifically, for any $R>2z_0$,
 from the proof of Lemma \ref{lem:non-compact}, we can find a sequence $\{\xi_{n,R}\}_{n\in\mathbb{N}^+}\subset \mathcal{C}^\infty_0([R,\infty))$ which is bounded in $\mathbb{F}_R^\vartheta$ but is not pre-compact in $L^2((R,\infty),\vartheta\ud z)$.

 For a function $\xi:[R,\infty]\to \R$ with $R>2z_0$, we define $(\xi)_R^\Gamma:\Gamma\to \R$ via $(\xi)_R^{\Gamma}(z,k)=0$ for $(z,k)\in\cup_{k=1}^{m-1}I_k$, $(\xi)_R^{\Gamma}(z,m)=0$ for $z<R$, and $(\xi)_R^{\Gamma}(z,m)=\xi(z)$ for $z\ge R$. The function $(\xi)_R^\Gamma$ may be regarded as the zero extension of $\xi$ to the graph $\Gamma$; it coincides with $\xi$ on the $m$-th edge for $z\ge R$, and is zero elsewhere.
From the asymptotic behaviors of $\alpha_k$ and $T_k$ near $O_\infty$, as given in \eqref{eq:A} and \eqref{eq:T}, respectively, we can fix a
 sufficiently large $R>2z_0$ such that there exist constants $C_i:=C_i(R)>0$, $i=1,2,3,4$ satisfying $C_1 z\le \alpha_k(z)\le C_2 z$ and $C_3\le T_k(z)\le C_4$ for all $z\ge R$. 
Then comparing the definitions of $\bar{\mathbb{W}}^{1,2}_\gamma$ and $\mathbb{F}_{R}^\vartheta$, we conclude that there exist some constants $C_i=C_i(R)$, $i=5,6$, such that for any $\xi\in \mathcal{C}_0^\infty([R,\infty))$,
\begin{gather*}
C_5^{-1}\|(\xi)_{R}^\Gamma\|_{\bar{H}_\gamma}\le \|\xi\|_{L^{2}((R,\infty);\vartheta\ud z)}\le C_5\|(\xi)_{R}^\Gamma\|_{\bar{H}_\gamma},\\
 C_6^{-1}\|(\xi)_{R}^\Gamma\|_{\bar{\mathbb{W}}^{1,2}_\gamma}\le \|\xi\|_{\mathbb{F}_{R}^\vartheta}\le C_6\|(\xi)_{R}^\Gamma\|_{\bar{\mathbb{W}}^{1,2}_\gamma}.
\end{gather*}
These equivalences imply that the sequence $\{(\xi_{n,R})_{R}^{\Gamma}\}_{n\in\mathbb{N}^+}$ is bounded in $\bar{\mathbb{W}}^{1,2}_\gamma$ but it is not pre-compact in $\bar{H}_\gamma$, where $(\xi_{n,R})_{R}^{\Gamma}$ denotes the zero extension of $\xi_{n,R}$ to the graph $\Gamma$.
 The proof is completed.
\hfill$\square$

\section{Proof of main results}\label{S:Pf}
In this section, we present the proofs of Theorems \ref{thm:LDP-graph}, \ref{thm:MDP-graph}, and \ref{thm:LDP}, following the weak convergence approach for LDP (see, e.g., \cite{BD00} and \cite{MSZ21}). For ${N>0}$, we denote
\begin{gather*}
S_N:=\left\{\phi \in L^{2}(0, T; \mathcal{U}_0 )\mid \int_{0}^{T}\|\phi(s)\|_{\mathcal U_0}^{2} \mathrm{~d} s \leq N\right\},\\
\mathcal{A}_{N}:=\left\{v:[0,T]\times \Omega\to \mathcal{U}_0 \mid v \text{ is } \{\mathscr{F}_{t}\} \text{-predictable and } v(\omega) \in S_{N} \text{ for } \mathbb{P}\text{-a.s. }\omega \right\}.
\end{gather*}
Unless otherwise specified, $S_N$ is endowed with the weak topology such that it is a Polish space. 
Note that
the sample path of $\mathcal{W}$ take values in $\mathcal{C}([0,T];\mathcal{U}_1)$ almost surely, where $\mathcal{U}_1$ is another Hilbert space such that the embedding $\mathcal{U}_0\subset \mathcal{U}_1$ is Hilbert--Schmidt. 
\begin{proposition}\cite[Theorem 3.2]{MSZ21} 
\label{prop:criterionMSZ}
Let $\mathcal E$ be a Polish space with the metric $\rho_{\mathcal E}$ and let $\{\zeta(\epsilon)\}_{\epsilon>0}\subset (0,\infty)$ satisfy that $\lim_{\epsilon\to 0}\zeta(\epsilon)=0$. For each $\epsilon \ge 0$, assume that $\mathcal{G}^{\epsilon}: \mathcal C([0, T], \mathcal{U}_1) \rightarrow \mathcal{E}$ is a measurable map and the following \textsl{Condition A} or \textsl{Condition A${^\prime}$} holds:

\textsl{Condition A.} For every $N >0$, the set $\left\{\mathcal{G}^{0}\left(\int_{0}^\cdot \phi(s) \mathrm{d} s\right): \phi\in S_{N}\right\}$ is a compact subset of $\mathcal{E}$;
For every $N >0$, if $\{v^{\epsilon}\}_{\epsilon>0} \subset \mathcal{A}_{N}$ converges in distribution (as $S_{N}$-valued random elements) to $v$, then $$\mathcal{G}^{\epsilon}\left(\mathcal{W}(\cdot)+\frac{1}{\sqrt{\zeta(\epsilon)}} \int_{0}^\cdot v^{\epsilon}(s) \mathrm{d} s\right)\rightarrow\mathcal{G}^{0}\left(\int_{0}^\cdot v(s) \mathrm{d} s\right)\quad \mbox{in distribution as } \epsilon\to 0.$$

\textsl{Condition A${^\prime}$.} For every $N >0$, if $\{\phi^{\epsilon}\}_{\epsilon>0} \subset S_{N}$ converges weakly to $\phi$, then as $\epsilon\to 0$,
		$\mathcal{G}^0\left(\int_{0}^\cdot \phi^{\epsilon}(s) \mathrm{d} s\right)\rightarrow\mathcal{G}^{0}\left(\int_{0}^\cdot \phi(s) \mathrm{d} s\right)$ in $\mathcal{E}$;
	 For every $N >0$, any family $\{v^\epsilon\}_{\epsilon>0}\subset \mathcal{A}_N$ and any $\delta_0>0$,
$$\lim_{\epsilon\to 0}\mathbb{P}\left\{\rho_{\mathcal{E}}\Big(\mathcal{G}^\epsilon(\mathcal{W}(\cdot)+\frac{1}{\sqrt{\zeta(\epsilon)}}\int_0^{\cdot}v^\epsilon(s)\ud s),\mathcal{G}^0(\int_0^\cdot v^\epsilon(s)\ud s)\Big)>\delta_0 \right\}=0.$$
		Then $\{\mathcal{G}^{\epsilon}(\mathcal{W})\}_{\epsilon>0}$ satisfies an LDP on $\mathcal{E}$ as $\epsilon\to 0$ with the speed $\zeta(\epsilon)^{-1}$ and the good rate function $I$ given by
	\begin{equation*}
		I(x)=\inf_{\left\{\phi\in L^2(0,T;\mathcal{U}_0),\,x=\mathcal{G}^0\left(\int_0^\cdot \phi(s)\ud s\right)\right\}}\frac12 \int_0^T\|\phi(s)\|_{\mathcal{U}_0}^2\ud s,\quad x\in\mathcal E.
	\end{equation*}
\end{proposition}
We remark that although \textsl{Condition A$^\prime$}
is stronger than \textsl{Condition A}, it is often easier to verify in certain cases. In particular, \textsl{Condition A$^\prime$} (resp. \textsl{Condition A}) will be used in proving Theorems \ref{thm:LDP-graph} and \ref{thm:MDP-graph} (resp. Theorem \ref{thm:LDP}).

\subsection{Compactness}\label{S:skeleton}
In this subsection, we present some quantitative properties of the skeleton equations \eqref{eq.skeleton-eq} and \eqref{eq:Nskeleton}. 
Define a convolution operator $\bar{\Xi}$ via
\begin{align*}
	\bar{\Xi}(y)(t) := \int_0^t \bar{S}(t-s)y(s)^\wedge\ud s, \qquad t\in[0,T],\quad y\in L^2(0,T; H_\gamma ).
\end{align*}
By \eqref{Stbound} and the Cauchy--Schwarz inequality, for any $y\in L^2(0,T; H_\gamma )$,
\begin{equation}\label{eq:Xibound}
	\|\bar{\Xi}(y)(t)\|_{\bar{H}_\gamma}\le C\int_0^t \|y(s)^\wedge\|_{\bar{H}_\gamma}\ud s\le C\|y\|_{L^2(0,T; H_\gamma )}\quad\forall~ t\in[0,T]
\end{equation}
which indicates that $\bar{\Xi}:L^2(0,T; H_\gamma )\to \mathcal C([0,T],\bar{H}_{\gamma})$ is a bounded linear operator.

\begin{lemma}\label{lemma:finite}
(1) Under Assumption \ref{asp:gamma}, for any fixed $N >0$, the set
	$\{\bar{\Xi}(y):\|y\|_{L^2(0,T; H_{\gamma})}\le N\}$ is a subset of $ \Lambda_{M,{\gamma}}$ (see \eqref{eq:LambdaM}) for some $M >0$.\\
(2) Under Assumption \ref{asp:compact}, $\bar{\Xi}$ is a compact operator from $L^2(0,T;\! H_{\gamma_1} )$ to $\mathcal C([0,T],\!\bar{H}_{\gamma})$.
 \end{lemma}

\begin{proof}
	(1) For the sake of simplicity, we denote in this proof $\bar{\Phi}:=\bar{\Xi}(y)$ and $$\bar{\Psi}(t):=\frac{1}{2}\sum_{k=1}^m\int_{I_k}\alpha_k(z)|\partial_z\bar{\Phi}(t,z,k)|^2\gamma(z,k)\ud z,\quad t\in[0,T],$$
	where the explicit dependence of $\bar{\Phi}$ and $\bar{\Psi}$ upon $y$ are omitted for simplicity. Since $\partial_t\bar{\Phi}(t)=\bar{L}\bar{\Phi}(t)+y(t)^\wedge$, $t\in[0,T]$,
	integrating by parts, together with \eqref{eq:barL} and the gluing condition \eqref{eq:glue} yields that for any $t\in[0,T]$,	\begin{align*}
		\frac{\ud}{\ud t}\bar{\Psi}(t)&=\sum_{k=1}^m\int_{I_k}\alpha_k(z)\gamma(z,k)\partial_z\bar{\Phi}(t,z,k)\partial_z\partial_t\bar{\Phi}(t,z,k)\ud z\\\notag		&=-2\sum_{k=1}^m\int_{I_k}T_k(z)\gamma(z,k)\bar{L}\bar{\Phi}(t,z,k)\partial_t\bar{\Phi}(t,z,k)\ud z\\\notag
		&\quad-\sum_{k=1}^m\int_{I_k}\alpha_k(z)\frac{\ud\gamma}{\ud z}(z,k)\partial_z\bar{\Phi}(t,z,k)\partial_t\bar{\Phi}(t,z,k)\ud z\\\notag
		&=-2\sum_{k=1}^m\int_{I_k}T_k(z)\gamma(z,k)|\partial_t\bar{\Phi}(t,z,k)|^2\ud z\\\notag &\quad+2\sum_{k=1}^m\int_{I_k}T_k(z)\gamma(z,k)y(t)^\wedge(z,k)\partial_t\bar{\Phi}(t,z,k)\ud z\\\notag
		&\quad-\sum_{k=1}^m\int_{I_k}\alpha_k(z)\frac{\ud\gamma}{\ud z}(z,k)\partial_z\bar{\Phi}(t,z,k)\partial_t\bar{\Phi}(t,z,k)\ud z.
			\end{align*}
Furthermore, the Young inequality leads to 
	\begin{align}\label{eq:Psi}
\frac{\ud}{\ud t}\bar{\Psi}(t)&\le -\sum_{k=1}^m\int_{I_k}T_k(z)\gamma(z,k)|\partial_t\bar{\Phi}(t,z,k)|^2\ud z\\\notag
		&\quad+2\sum_{k=1}^m\int_{I_k}T_k(z)|y(t)^\wedge(z,k)|^2\gamma(z,k)\ud z\\\notag	&\quad+\frac12\sum_{k=1}^m\int_{I_k}\frac{1}{T_k(z)\gamma^2(z,k)}\alpha^2_k(z)|\frac{\ud\gamma}{\ud z}(z,k)|^2\gamma(z,k)|\partial_z\bar{\Phi}(t,z,k)|^2\ud z.
	\end{align}
	Here, the last term on the right-hand side of \eqref{eq:Psi} is bounded by a multiple of $\bar{\Psi}(t)$ due to Assumption \ref{asp:gamma}. Consequently, there exists some constant $C>0$ such that for any $t\in[0,T]$,
	\begin{align*}
		&\frac{\ud}{\ud t}\bar{\Psi}(t)
		+\sum_{k=1}^m\int_{I_k}T_k(z)\gamma(z,k)|\partial_t\bar{\Phi}(t,z,k)|^2\ud z\le2\|y(t)^\wedge\|_{\bar{H}_{\gamma}}^2+C\bar{\Psi}(t).
	\end{align*}
	Applying the Gronwall inequality and using $\bar{\Phi}(0)=0$ (which implies $\bar{\Psi}(0)=0$) gives
	\begin{align*}
		&\sup_{t\in[0,T]}\bar{\Psi}(t)
		+\int_0^T\sum_{k=1}^m\int_{I_k}T_k(z)\gamma(z,k)|\partial_t\bar{\Phi}(t,z,k)|^2\ud z\ud t\le C\|y\|_{L^2(0,T; H_{\gamma} )}^2.
	\end{align*}
Combining this with \eqref{eq:Xibound} completes the proof of the first claim of the lemma.

(2) Under Assumption \ref{asp:compact}(i), 
for a fixed $N >0$, the set
	$\{\bar{\Xi}(y):\|y\|_{L^2(0,T; H_{\gamma_1})}\le N\}$ is contained in $ \Lambda_{M,{\gamma_1}}$ for some $M >0$, and is therefore a pre-compact set of $\mathcal C([0,T],\bar{H}_{\gamma})$ due to Assumption \ref{asp:compact}(ii). This proves the compactness of $\bar{\Xi}$ from $L^2(0,T; H_{\gamma_1} )$ to $\mathcal C([0,T],\bar{H}_{\gamma})$.
\end{proof}

\begin{remark}\label{rem:strong-continuity}
The map $\phi\mapsto \bar{Z}^\phi$ is continuous from $L^2(0,T;\mathcal{U}_0)$ to $\mathcal{C}([0,T];\bar{H}_\gamma)$.
To illustrate this fact, we fix $\phi_1\in L^2(0,T;\mathcal{U}_0)$ and let $\phi_2\in L^2(0,T;\mathcal{U}_0)$ with $\|\phi_1-\phi_2\|_{L^2(0,T;\mathcal{U}_0)}\le 1$. Then it follows from \eqref{eq:ut} and \eqref{Stbound} that for any $t\in[0,T]$,
	\begin{align*}
		\|\bar{Z}^{\phi_1}(t)-\bar{Z}^{\phi_2}(t)\|_{\bar{H}_\gamma}&\le \int_0^t\|B(\bar{Z}^{\phi_1}(s))-B(\bar{Z}^{\phi_2}(s))\|_{\bar{H}_\gamma}\ud s\\
		&\quad+ \int_0^t\|(G(\bar{Z}^{\phi_1}(s))-G(\bar{Z}^{\phi_2}(s)))\phi_2(s)^\wedge\|_{\bar{H}_\gamma}\ud s\\
		&\quad+\int_0^t\|G(\bar{Z}^{\phi_1}(s))(\phi_1(s)^\wedge-\phi_2(s)^\wedge)\|_{\bar{H}_\gamma}\ud s.
	\end{align*}
	Utilizing the Cauchy--Schwarz inequality, the Lipschitz continuity of $B$ and $G$ in \eqref{eq:Glip}, as well as the linear growth of $G$ in \eqref{eq:Glingr}, one has 
	\begin{align*}
		\|\bar{Z}^{\phi_1}(t)-\bar{Z}^{\phi_2}(t)\|^2_{\bar{H}_\gamma}&\le C(1+\|\phi_2\|_{L^2(0,T;\mathcal U_0)}^2)\int_0^t\|\bar{Z}^{\phi_1}(s)-\bar{Z}^{\phi_2}(s)\|_{\bar{H}_\gamma}^2\ud s\\
		&\quad+C\int_0^t(1+\|\bar{Z}^{\phi_1}(s)\|^2_{\bar{H}_\gamma})\ud s\|\phi_1-\phi_2\|_{L^2(0,T;\mathcal U_0)}^2\\
		&\le C(1+\|\phi_1\|_{L^2(0,T;\mathcal U_0)}^2)\int_0^t\|\bar{Z}^{\phi_1}(s)-\bar{Z}^{\phi_2}(s)\|_{\bar{H}_\gamma}^2\ud s\\
		&\quad+C(\|\phi_1\|_{L^2(0,T;\mathcal{U}_0)},\| \mathfrak{u} \|_{ H_\gamma },T)\|\phi_1-\phi_2\|_{L^2(0,T;\mathcal U_0)}^2
	\end{align*}
	where in the last step we used \eqref{eq:Zvt} and $\|\phi_1-\phi_2\|_{L^2(0,T;\mathcal{U}_0)}\le 1$. Therefore, similar to the proof of \eqref{eq:Zvt}, we derive by using the Gronwall inequality that
	\begin{align*}
		\|\bar{Z}^{\phi_1}-\bar{Z}^{\phi_2}\|^2_{\mathcal{C}([0,T];\bar{H}_\gamma)}&\le C(\|\phi_1\|_{L^2(0,T;\mathcal{U}_0)},\| \mathfrak{u} \|_{ H_\gamma },T)\|\phi_1-\phi_2\|_{L^2(0,T;\mathcal U_0)}^2,
	\end{align*}
	which shows that the map $\phi\mapsto \bar{Z}^\phi$ is continuous from $L^2(0,T;\mathcal{U}_0)$ to $\mathcal{C}([0,T];\bar{H}_\gamma)$.
\end{remark}

We next prove the continuity of the solution maps for the skeleton equations \eqref{eq.skeleton-eq} and \eqref{eq:Nskeleton} from $L_{weak}^2(0,T;\mathcal{U}_0)$ to $\mathcal{C}([0,T];\bar{H}_\gamma)$, where $L_{weak}^2(0,T;\mathcal{U}_0)$ denotes the space $L^2(0,T;\mathcal{U}_0)$ endowed with the weak convergence. 

\begin{lemma}\label{lem:compact}
Assume that $b$ and $g$ are globally Lipschitz continuous and $\mathfrak{u}\in H_\gamma\cap H_{\gamma_1}$.
Let Assumptions \ref{asp:gamma} and \ref{asp:compact} hold.
If $\{\phi^{n}\}_{n\in\mathbb{N}^+} \subset S_{N}$ converges weakly to $\phi$ as $n\to \infty$,
then $\bar{Z}^{\phi_n}$ (resp. $\bar{R}^{\phi_n}$) converges in $\mathcal C([0,T],\bar{H}_\gamma)$ to $\bar{Z}^{\phi}$ (resp. $\bar{R}^{\phi}$) as $n\to \infty$.
\end{lemma}

\begin{proof}
We only prove the convergence of $\bar{Z}^{\phi_n}$ since that of $\bar{R}^{\phi_n}$ can be proved in a similar manner.
Due to \eqref{eq:ut} and \eqref{Stbound}, it holds that for any $t\in[0,T]$,
\begin{align*}
\|\bar{Z}^{\phi}(t)-\bar{Z}^{\phi_n}(t)\|_{\bar{H}_\gamma}&\le \int_0^t \|B(\bar{Z}^{\phi}(s))-B(\bar{Z}^{\phi_n}(s))\|_{\bar{H}_\gamma}\ud s\\
&\quad+\int_0^t \|(G(\bar{Z}^{\phi}(s))-G(\bar{Z}^{\phi_n}(s)))\phi_n (s)^\wedge\|_{\bar{H}_\gamma}\ud s\\
&\quad+\left\|\int_0^t \bar{S}(t-s) G(\bar{Z}^{\phi}(s))(\phi(s)-\phi_n(s))^\wedge\ud s\right\|_{\bar{H}_\gamma}.
\end{align*}
By the Cauchy--Schwarz inequality, $\{\phi_n\}_{n\in\mathbb{N}^+}\subset S_N$, and
the Lipschitz continuity of $B$ and $G$ in \eqref{eq:Glip}, we infer that for any $t\in[0,T]$,
\begin{align*}
\|\bar{Z}^{\phi}(t)-\bar{Z}^{\phi_n}(t)\|^2_{\bar{H}_\gamma}&\le C(1+N)\int_0^t \|\bar{Z}^{\phi}(s)-\bar{Z}^{\phi_n}(s)\|^2_{\bar{H}_\gamma}\ud s\\
&\quad+C\sup_{t\in[0,T]}\left\|\int_0^t \bar{S}(t-s) G(\bar{Z}^{\phi}(s))(\phi(s)-\phi_n(s))^\wedge\ud s\right\|_{\bar{H}_\gamma}^2.
\end{align*}
Subsequently, an application of the Gronwall inequality leads to
\begin{equation}\label{eq:U-U-H}
\|\bar{Z}^{\phi}(t)-\bar{Z}^{\phi_n}(t)\|^2_{\mathcal{C}([0,T];\bar{H}_\gamma)}\le C\sup_{t\in[0,T]}\left\|\int_0^t \bar{S}(t-s) G(\bar{Z}^{\phi}(s))(\phi_n(s)-\phi(s))^\wedge\ud s\right\|_{\bar{H}_{\gamma}},
\end{equation}
where the constant $C>0$ may depend on $N$.
In virtue of $\mathfrak{u}\in H_{\gamma_1}$, \eqref{eq:Zvt}, and \eqref{fvarphi}, one has that for any $y\in L^2(0,T; H_{\gamma_1} )$,
\begin{align*}
	\|G(\bar{Z}^{\phi}(s)^\vee)^*y\|^2_{L^2(0,T;\mathcal{U}_0)}&\le \int_0^T\|G(\bar{Z}^{\phi}(t))^\vee\|_{\mathcal{L}_2(\mathcal{U}_0, H_{\gamma_1} )}^2\| y(t)\|^2_{ H_{\gamma_1} }\ud t\\
	&\le C(1+\|\bar{Z}^{\phi}\|^2_{\mathcal{C}([0,T];\bar{H}_{\gamma_1})})\|y\|^2_{L^2(0,T;\bar{H}_{\gamma_1})} <\infty.
\end{align*}
It follows from $\phi_n\rightharpoonup\phi$ in $L^2(0,T;\mathcal{U}_0)$ that for any $y\in L^2(0,T; H_{\gamma_1} )$,
\begin{align*}
	\langle G((\bar{Z}^{\phi})^\vee)\phi_n,y\rangle_{L^2(0,T; H_{\gamma_1} )}&=\langle \phi_n,G((\bar{Z}^{\phi})^\vee)^*y\rangle_{L^2(0,T;\mathcal{U}_0)}\\
	&\to \langle \phi,G((\bar{Z}^{\phi})^\vee)^*y\rangle_{L^2(0,T;\mathcal{U}_0)}=\langle G((\bar{Z}^{\phi})^\vee)\phi,y\rangle_{L^2(0,T; H_{\gamma_1} )},
\end{align*}
which means that 
$G((\bar{Z}^{\phi})^\vee)\phi_n\rightharpoonup G((\bar{Z}^{\phi})^\vee)\phi$ in $L^2(0,T; H_{\gamma_1} )$ as $n\to\infty$. Furthermore, observe that for any $\varphi\in L^2(0,T;\mathcal{U}_0)$,
\begin{equation*}
	\int_0^t \bar{S}(t-s) G(\bar{Z}^{\phi}(s))\varphi(s)^\wedge\ud s=\bar{\Xi}\big((G(\bar{Z}^{\phi})^\vee)\varphi\big)(t),\quad t\in[0,T].
\end{equation*}
Finally, the weak convergence of $G((\bar{Z}^{\phi})^\vee)\phi_n$ combined with
the compactness of $\bar{\Xi}$ from $L^2(0,T; H_{\gamma_1} )$ to $\mathcal C([0,T],\bar{H}_{\gamma})$ (see Lemma \ref{lemma:finite}(2)) and \eqref{eq:U-U-H} finishes the proof. 
\end{proof}

Since Lemma \ref{lem:compact} only involve the a priori $L^2$-estimates, Assumption \ref{asp:gamma} in Lemma \ref{lem:compact} can be replaced by Assumptions \ref{Asp:SG-strong} and \ref{asp:S-S} (see Remark \ref{rem:L2}).

\subsection{Proof of Theorem \ref{thm:LDP-graph}}
In this subsection, we prove Theorem \ref{thm:LDP-graph} which concerns the LDP of $\{(\bar{u}_\epsilon)^\vee\}_{\epsilon\in(0,1]}$ on $\mathcal{C}([0,T];H_\gamma)$ as $\epsilon\to0$.

Let $\mathcal G^\epsilon$ be the measurable map associating $\mathcal{W}$ to $(\bar{u}_\epsilon)^\vee$, where $\bar{u}_\epsilon$ is the solution to \eqref{eq:intro}, i.e., $(\bar{u}_\epsilon)^\vee=\mathcal G^\epsilon(\mathcal{W})$ for $\epsilon>0$.
For any $v\in\mathcal A_N$ and $\epsilon>0$, the Girsanov theorem (see, e.g., \cite[Theorem 10.14]{DZ14}) indicates that $\widetilde{\mathcal{W}}^{\epsilon,v}:=\mathcal{W}+\frac{1}{\sqrt\epsilon}\int_0^\cdot v(s)\ud s$ is a $\mathcal{U}_0$-cylindrical Wiener process under $\widetilde{\mathbb{P}}^{\epsilon,v}$, where
\begin{equation}\label{eq:P} \frac{\ud\widetilde{\mathbb{P}}^{\epsilon,v}}{\ud\mathbb{P}}:=\exp\left(-\frac{1}{\sqrt\epsilon}\int_0^T\langle v(s),\ud \mathcal{W}(s)\rangle_{\mathcal{U}_0}-\frac{1}{2\epsilon}\int_0^T\|v(s)\|_{\mathcal{U}_0}^2\ud s\right),\quad \epsilon>0.
\end{equation}
Hence,
 $\mathcal G^\epsilon(\widetilde{\mathcal{W}}^{\epsilon,v})=(\bar{\mathfrak{U}}_{\epsilon}^v)^\vee$ where $\bar{\mathfrak{U}}_{\epsilon}^v$ is the unique mild solution to \eqref{eq:intro} under $\widetilde{\mathbb{P}}^{\epsilon,v}$ with $(\bar{u}^\epsilon,\bar{\mathcal{W}})$ replaced by $(\bar{\mathfrak{U}}_{\epsilon}^v,(\widetilde{\mathcal{W}}^{\epsilon,v})^\wedge)$. Since $\mathbb{P}$ is equivalent to $\widetilde{\mathbb{P}}^{\epsilon,v}$, $\bar{\mathfrak{U}}_{\epsilon}^v$ satisfies the following stochastic controlled equation associated with \eqref{eq:intro}
\begin{align*} 
	\ud \bar{\mathfrak{U}}_{\epsilon}^v(t)=\bar{L} \bar{\mathfrak{U}}_{\epsilon}^v(t)\ud t
	+B(\bar{\mathfrak{U}}_{\epsilon}^v(t))\ud t+\sqrt{\epsilon}G(\bar{\mathfrak{U}}_{\epsilon}^v(t))\ud\bar{\mathcal{W}}(t)+G(\bar{\mathfrak{U}}_{\epsilon}^v(t))v(t)^\wedge\ud t
\end{align*}
for $t\in[0,T]$
with $\bar{\mathfrak{U}}_{\epsilon}^v(0) = \mathfrak{u}^\wedge$, under $\mathbb{P}$. Since $B:\bar{H}_\gamma\to \bar{H}_\gamma$ and $G:\bar{H}_\gamma\to \mathcal{L}_2(\bar{\mathcal{U}}_0,\bar{H}_\gamma)$ are globally Lipschitz continuous (see \eqref{eq:Glip}), it can be shown by using \eqref{Stbound} that
for any $p\ge1$,
\begin{equation}\label{eq:Svep}
	\E\left[\|\bar{\mathfrak{U}}_{\epsilon}^v(t)\|_{\bar{H}_\gamma }^p\right]\le C(T,p,\| \mathfrak{u} \|_{ H_\gamma },N)\quad\forall~\epsilon\in(0,1],t\in[0,T],v\in \mathcal{A}_N.
\end{equation}
In addition, we denote \begin{equation}\label{eq:G0}
	\mathcal G^0\left(\int_0^\cdot \phi(s)\ud s\right):=(\bar{Z}^\phi(\cdot))^\vee,\quad \phi\in L^2(0,T;\mathcal{U}_0),
\end{equation}
where $\bar{Z}^\phi(\cdot)$ is given by the skeleton equation \eqref{eq.skeleton-eq}.

\textbf{Proof of Theorem \ref{thm:LDP-graph}.} To prove Theorem \ref{thm:LDP-graph}, we apply Proposition \ref{prop:criterionMSZ} with \textsl{Condition A}${^\prime}$ and $\zeta(\epsilon)=\epsilon$. Recalling $\mathcal G^\epsilon(\widetilde{\mathcal{W}}^{\epsilon,v})=(\bar{\mathfrak{U}}_{\epsilon}^v)^\vee$ and $\mathcal{G}^0$ defined in \eqref{eq:G0}, the proof of Theorem \ref{thm:LDP-graph} therefore reduces to verifying the following statements \textbf{(A1)} and \textbf{(A2)}.
\begin{enumerate}
	\item[\textbf{(A1)}] For every $N >0$, if $\{\phi^{\epsilon}\}_{\epsilon>0} \subset S_{N}$ converges weakly to $\phi$, then $\bar{Z}^{\phi^\epsilon}\to \bar{Z}^\phi$ in $\mathcal{C}([0,T];\bar{H}_\gamma)$
	as $\epsilon\to 0$;

	\item[\textbf{(A2)}] For every $N >0$, any family $\{v^\epsilon\}_{\epsilon>0}\subset \mathcal{A}_N$ and any $\delta_0>0$,
	$$\lim_{\epsilon\to 0}\mathbb{P}\left\{\sup_{t\in[0,T]}\|\bar{\mathfrak{U}}_{\epsilon}^{v^\epsilon}(t)-\bar{Z}^{v^\epsilon}(t)\|_{\bar{H}_\gamma}>\delta_0 \right\}=0.$$
\end{enumerate}
The above condition \textbf{(A1)} has been shown in Lemma \ref{lem:compact}, while \textbf{(A2)} comes from Lemma \ref{lem:Y-Z} and the Markov inequality.
\hfill$\square$

\begin{lemma}\label{lem:Y-Z}
Assume that $b$ and $g$ are globally Lipschitz continuous and $\mathfrak{u}\in H_\gamma$.
Then under Assumption \ref{asp:gamma}, for any $\{v^\epsilon\}_{\epsilon\in(0,1]}\subset \mathcal{A}_N$ with some $N >0$,
 \begin{equation*} \E\bigg[\sup_{t\in[0,T]}\|\bar{\mathfrak{U}}_{\epsilon}^{v^\epsilon}(t)-\bar{Z}^{v^\epsilon}(t)\|^2_{ H_\gamma }\bigg]\le C(T,\| \mathfrak{u} \|_{ H_\gamma },N)\epsilon.
\end{equation*}
\end{lemma}
\begin{proof}
Since $\{v^\epsilon\}_{\epsilon\in(0,1]}\subset \mathcal{A}_N$, analogue to the proof of \eqref{eq:Zphi}, it holds that for any $t_1\in[0,T]$,
\begin{align}\label{eq:Ygronwall}\notag
\E\bigg[\sup_{t\in[0,t_1]}\|\bar{\mathfrak{U}}_{\epsilon}^{v^\epsilon}(t)-\bar{Z}^{v^\epsilon}(t)\|^2_{\bar{H}_\gamma }\bigg]&\le C(N+1)\E\bigg[\int_0^{t_1}\sup_{r\in[0,s]}\|\bar{\mathfrak{U}}_{\epsilon}^{v^\epsilon}(r)-\bar{Z}^{v^\epsilon}(r)\|^2_{ H_\gamma }\ud s\bigg]\\
&\quad+\epsilon \E\left[\sup_{t\in[0,T]}\Big\|\int_0^t \bar{S}(t-s)G(\bar{\mathfrak{U}}_{\epsilon}^{v^\epsilon}(s))\ud\bar{\mathcal{W}}(s)\Big\|_{ H_\gamma }^2\right].
\end{align}
Using the factorization argument \cite[Proposition 7.3]{DZ14}, for any $p>2$,
\begin{align*}
\E\left[\sup_{t\in[0,T]}\Big\|\int_0^t \bar{S}(t-s)G(\bar{\mathfrak{U}}_{\epsilon}^{v^\epsilon}(s))\ud\bar{\mathcal{W}}(s)\Big\|_{ H_\gamma }^p\right]&\le C\int_0^{T}\E\left[\|G(\bar{\mathfrak{U}}_{\epsilon}^{v^\epsilon}(s))\|_{\mathcal{L}_2(\bar{\mathcal{U}}_0, \bar{H}_\gamma )}^p\right]\ud s\\
&\le C(p,T,\| \mathfrak{u} \|_{ H_\gamma },N),
\end{align*}
thanks to \eqref{eq:Svep} and the linear growth of $G$ in \eqref{eq:Glingr}. By the H\"older inequality, the same estimate also holds for $p=2$.
Finally, applying the Gronwall inequality to \eqref{eq:Ygronwall} leads to the desired result.
\end{proof}

\subsection{Proof of Theorem \ref{thm:LDP}}
This subsection is devoted to proving Theorem \ref{thm:LDP} concerning the LDP of $\{u^\epsilon_{\psi(\epsilon)}\}_{\epsilon\in(0,1]}$ on $\mathcal{C}([\tau_0,T];H_\gamma)$ as $\epsilon\to 0$, for any fixed $\tau_0\in(0,T)$. 
By \eqref{eq:intro2}, we have
\begin{equation}\label{eq:model-dia}
\ud u^\epsilon_{\psi(\epsilon)}(t)=L_{\psi(\epsilon)} u^\epsilon_{\psi(\epsilon)}(t)\ud t
 +B(u^\epsilon_{\psi(\epsilon)}(t))\ud t+\sqrt{\epsilon}G(u^\epsilon_{\psi(\epsilon)}(t))\ud\mathcal{W}(t),\quad t\in[0,T]
\end{equation}
with the initial value $u^\epsilon_{\psi(\epsilon)}(0)=\mathfrak{u}$. 
For $\epsilon>0$,
let $\mathcal F^\epsilon:\mathcal{C}([0,T];\mathcal{U}_1)\mapsto\mathcal{C}([0,T];H_\gamma)$ be the measurable map associating $\mathcal{W}$ to the mild solution $u^\epsilon_{\psi(\epsilon)}$ to \eqref{eq:model-dia}, i.e., $u^\epsilon_{\psi(\epsilon)}=\mathcal F^\epsilon(\mathcal{W})$. Referring to \eqref{eq:P} and $\widetilde{\mathcal{W}}^{\epsilon,v}:=\mathcal{W}+\frac{1}{\sqrt\epsilon}\int_0^\cdot v(s)\ud s$, one has that
 $\mathfrak{L}_{\epsilon}^v:=\mathcal F^\epsilon(\widetilde{\mathcal{W}}^{\epsilon,v})$ 
 is the unique mild solution to \eqref{eq:model-dia} under $\widetilde{\mathbb{P}}^{\epsilon,v}$ with $(u^\epsilon_{\psi(\epsilon)},\mathcal{W})$ replaced by $(\mathfrak{L}_{\epsilon}^v,\widetilde{\mathcal{W}}^{\epsilon,v})$. Since $\mathbb{P}$ is equivalent to $\widetilde{\mathbb{P}}^{\epsilon,v}$, $\mathfrak{L}_{\epsilon}^v$ satisfies the following stochastic controlled equation associated with \eqref{eq:model-dia}
\begin{equation} \label{eq.controlled-eq}
\ud \mathfrak{L}_{\epsilon}^v(t)= L_{\psi(\epsilon)} \mathfrak{L}_{\epsilon}^v(t)\ud t
 +B(\mathfrak{L}_{\epsilon}^v(t))\ud t+\sqrt{\epsilon}G(\mathfrak{L}_{\epsilon}^v(t))\ud\mathcal{W}(t)+G(\mathfrak{L}_{\epsilon}^v(t))v(t)\ud t
\end{equation}
for $t\in[0,T]$
with $\mathfrak{L}_{\epsilon}^v(0) = \mathfrak{u} $, under $\mathbb{P}$. Since $B:H_\gamma\to H_\gamma$ and $G:H_\gamma\to \mathcal{L}_2(\mathcal{U}_0,H_\gamma)$ are globally Lipschitz continuous, it can be shown by using Assumption \ref{Asp:SG-strong} that
for any $p\ge1$,
\begin{equation*}
	\E\left[\|\mathfrak{L}_{\epsilon}^v(t)\|_{ H_\gamma }^p\right]\le C(T,p,\| \mathfrak{u} \|_{ H_\gamma },N)\quad\forall~\epsilon\in(0,1],t\in[0,T],v\in \mathcal{A}_N.
\end{equation*}
Set $(\bar{Z}^\phi(\cdot))^\vee:=\mathcal{F}^{0}\left(\int_{0}^\cdot \phi(s) \mathrm{d} s\right)$, where $\bar{Z}^\phi$ is defined by \eqref{eq.skeleton-eq} for $\phi\in L^2(0,T;\mathcal{U}_0)$.

 In the following, let $\{v^{\epsilon}\}_{\epsilon>0} \subset \mathcal{A}_{N}$ converge in distribution (as $S_{N}$-valued random elements) to $v$. 
To measure the error between $\mathfrak{L}_{\epsilon}^{v^\epsilon}(\cdot)$ and $(\bar{Z}^v(\cdot))^\vee$, we introduce an auxiliary equation
\begin{align} \label{eq.wcontrolled-eq}
\ud w_{\epsilon}^{\phi}(t)= L_{\psi(\epsilon)} w_{\epsilon}^{\phi}(t)\ud t
 +B(w_{\epsilon}^{\phi}(t))\ud t+G(w_{\epsilon}^{\phi}(t)){\phi}(t)\ud t, \quad t\in[0,T]
\end{align}
with $w_{\epsilon}^{\phi}(0)= \mathfrak{u} $,
where $\phi\in L^2(0,T;\mathcal{U}_0)$ and $\epsilon>0$. 
 Proceeding as in the proofs of \eqref{eq:ut} and \eqref{eq:Zvt},
by Assumption \ref{Asp:SG-strong}, one also has 
that for each fixed $\epsilon>0$, there exists a unique mild solution $w_\epsilon^{\phi}$ to \eqref{eq.wcontrolled-eq} in $ \mathcal C([0, T] ; H_\gamma )$ satisfying
\begin{equation}\label{eq:wphit}
	\|w^{\phi}_\epsilon\|_{\mathcal{C}([0,T]; H_\gamma )}\le C(\|\phi\|_{L^2(0,T;\mathcal{U}_0)},\| \mathfrak{u} \|_{ H_\gamma },T),\quad \epsilon>0.
\end{equation}
Moreover, the map $\phi\mapsto w^{\phi}_\epsilon$ is continuous from $L^2(0,T;\mathcal{U}_0)$ to $\mathcal{C}([0,T]; H_\gamma )$ (see Remark \ref{rem:strong-continuity} for a similar argument).

\begin{lemma}\label{lem:w-u}
Assume that $b$ and $g$ are globally Lipschitz continuous and $\mathfrak{u}\in H_\gamma$. Then under Assumption \ref{Asp:SG-strong}, for any $\{v^\epsilon\}_{\epsilon>0}\subset \mathcal{A}_N$ with some $N >0$,
\begin{equation*}
\E\bigg[\sup_{t\in[0,T]}\|w_\epsilon^{v^\epsilon}(t)-\mathfrak{L}_{\epsilon}^{v^\epsilon}(t)\|^2_{ H_\gamma }\bigg]\le C(T,\| \mathfrak{u} \|_{ H_\gamma },N)\epsilon.
\end{equation*}
\end{lemma}
The proof is analogous to that of Lemma \ref{lem:Y-Z} and is omitted for brevity.

\begin{lemma}\label{lem:w-Z}
	Assume that $b$ and $g$ are globally Lipschitz continuous and $\mathfrak{u}\in H_\gamma\cap H_{\gamma_1}$. Let Assumptions \ref{Asp:SG-strong}, \ref{asp:S-S}, and \ref{asp:compact} hold.
 For every $N >0$, if $\{v^{\epsilon}\}_{\epsilon>0} \subset \mathcal{A}_{N}$ converges a.s. (as $S_{N}$-valued random elements) to $v$, then for any $0<\tau_0<T$,
\begin{equation*}
\lim_{\epsilon\to 0}\E\bigg[\sup_{t\in[\tau_0,T]}\|w_\epsilon^{v^\epsilon}(t)-\bar{Z}^{v}(t)^\vee\|_{ H_\gamma }^2\bigg]=0.
\end{equation*}
\end{lemma}
\begin{proof}
By \eqref{eq.controlled-eq} and \eqref{eq:ut},
\begin{align}\label{eq:w-v}\notag
w_\epsilon^{v^\epsilon}(t)-\bar{Z}^{v}(t)^\vee&=S_{\psi(\epsilon)}(t) \mathfrak{u} -\bar{S}(t)^\vee \mathfrak{u} \\\notag
&+\int_0^t S_{\psi(\epsilon)}(t-s)B(w_\epsilon^{v^\epsilon}(s))\ud s-\int_0^t\bar{S}(t-s)^\vee B(\bar{Z}^{v}(s)^\vee)\ud s\\\notag
&+\int_0^t S_{\psi(\epsilon)}(t-s)G(w_\epsilon^{v^\epsilon}(s))v^\epsilon(s)\ud s-\int_0^t\bar{S}(t-s)^\vee G(\bar{Z}^{v}(s)^\vee)v(s)\ud s\\
&=: I_{\epsilon,1}(t)+I_{\epsilon,2}(t)+I_{\epsilon,3}(t).
\end{align}
We further decompose $I_{\epsilon,2}(t)$ and $I_{\epsilon,3}(t)$ as follows
\begin{align*}
I_{\epsilon,2}(t)
&=\int_0^t S_{\psi(\epsilon)}(t-s)(B(w_\epsilon^{v^\epsilon}(s))-B(\bar{Z}^{v}(s)^\vee))\ud s\\
&\quad+\int_0^t (S_{\psi(\epsilon)}(t-s)-\bar{S}(t-s)^\vee) B(\bar{Z}^{v}(s)^\vee)\ud s=:I_{\epsilon,2,1}(t)+I_{\epsilon,2,2}(t),\\
I_{\epsilon,3}(t)&=\int_0^t S_{\psi(\epsilon)}(t-s) (G(w_\epsilon^{v^\epsilon}(s))-G(\bar{Z}^{v}(s)^\vee))v^\epsilon(s)\ud s\\
&\quad+\int_0^t (S_{\psi(\epsilon)}(t-s)- \bar{S}(t-s)^\vee) G(\bar{Z}^{v}(s)^\vee)v^\epsilon(s)\ud s\\
&\quad+\int_0^t \bar{S}(t-s)^\vee G(\bar{Z}^{v}(s)^\vee)(v^\epsilon(s)-v(s))\ud s=:I_{\epsilon,3,1}(t)+I_{\epsilon,3,2}(t)+I_{\epsilon,3,3}(t).
\end{align*}
Let $\tau_1\in(0,T)$ be arbitrarily fixed.
 Invoking Assumption \ref{Asp:SG-strong}, the Lipschitz continuity of $B$ and $G$, as well as the fact $v^\epsilon\in\mathcal{A}_N$, we obtain that for any $t\in[\tau_1,T]$,
\begin{align*}
	&\|I_{\epsilon,2,1}(t)\|^2_{ H_\gamma }+\|I_{\epsilon,3,1}(t)\|^2_{ H_\gamma }\le C(T,N)\int_{0}^{t}\|w_\epsilon^{v^\epsilon}(s)-\bar{Z}^{v}(s)^\vee\|^2_{ H_\gamma }\ud s \\
	&\le C(T,N)\int_{\tau_1}^{t}\sup_{r\in[\tau_1,s]}\|w_\epsilon^{v^\epsilon}(r)-\bar{Z}^{v}(r)^\vee\|^2_{ H_\gamma }\ud s + C(N,\| \mathfrak{u} \|_{ H_\gamma }, T)\tau_1,\quad \textup{a.s.},
\end{align*}
where \eqref{eq:Zvt} and \eqref{eq:wphit} were used in the last step. Combining the above inequality with \eqref{eq:w-v} results in that for any $t_1\in[0,T]$,
\begin{align*}
	\E\bigg[\sup_{t\in[\tau_1,t_1]}\|w_\epsilon^{v^\epsilon}(t)-\bar{Z}^{v}(t)^\vee\|_{ H_\gamma }^2\bigg]&\le C(N,T)\int_{\tau_1}^{t_1}\E\bigg[\sup_{r\in[\tau_1,s]}\|w_\epsilon^{v^\epsilon}(r)-\bar{Z}^{v}(r)^\vee\|^2_{ H_\gamma }\bigg]\ud s \\ 
	&+C\sup_{t\in[\tau_1,T]}\|I_{\epsilon,1}(t)\|_{ H_\gamma }^2+C(N,\| \mathfrak{u} \|_{ H_\gamma }, T)\tau_1\\
	&+C\E\bigg[\sup_{t\in[\tau_1,T]}\left(\|I_{\epsilon,2,2}(t)+I_{\epsilon,3,2}(t)+I_{\epsilon,3,3}(t)\|^2_{ H_\gamma }\right)\bigg].
\end{align*}
This, together with the Gronwall inequality, gives that for any $\epsilon\in(0,1]$,
\begin{align}\label{eq:Errwz}
	&\E\bigg[\sup_{t\in[\tau_1,T]}\|w_\epsilon^{v^\epsilon}(t)-\bar{Z}^{v}(t)^\vee\|_{ H_\gamma }^2\bigg]\le C\sup_{t\in[\tau_1,T]}\|I_{\epsilon,1}(t)\|_{ H_\gamma }^2+C(N,\| \mathfrak{u} \|_{ H_\gamma }, T)\tau_1\\\notag	&\qquad\qquad\qquad\qquad\qquad+C\E\bigg[\sup_{t\in[\tau_1,T]}\left(\|I_{\epsilon,2,2}(t)+I_{\epsilon,3,2}(t)+I_{\epsilon,3,3}(t)\|^2_{ H_\gamma }\right)\bigg].
\end{align}
It follows from Assumption \ref{asp:S-S}, \eqref{eq:deltaepsilon}, and $ \mathfrak{u} \in H_\gamma $ that
\begin{equation}\label{eq:I1ep}
\lim_{\epsilon\to 0}	\sup_{t\in[\tau_1,T]}\|I_{\epsilon,1}(t)\|_{ H_\gamma }=0.
\end{equation}
It remains to estimate the terms related to $I_{\epsilon,2,2}$, $I_{\epsilon,3,2}$ and $I_{\epsilon,3,3}.$

\textit{\bf Estimate of $I_{\epsilon,2,2}$}.
The linear growth of $B$ and \eqref{eq:Zvt} yield that for $v\in\mathcal{A}_N$,
\begin{equation}\label{eq:BZ}
	\sup_{t\in[0,T]}\|B(\bar{Z}^{v}(s)^\vee)\|_{ H_\gamma }\le C(N,\| \mathfrak{u} \|_{ H_\gamma }, T),\quad \textup{a.s.}
\end{equation}
Hence by Assumption \ref{Asp:SG-strong}, \eqref{Stbound}, and \eqref{eq:Zvt},
for any $\tau_1<t\le T$,
\begin{align}\label{eq:I22}
\|I_{\epsilon,2,2}(t)\|_{ H_\gamma }^2&\le T\int_0^{t-\tau_1} \|(S_{\psi(\epsilon)}(t-s)-\bar{S}(t-s)^\vee) B(\bar{Z}^{v}(s)^\vee)\|_{ H_\gamma }^2\ud s\\\notag
&\quad+T\int_{t-\tau_1}^t \|(S_{\psi(\epsilon)}(t-s)-\bar{S}(t-s)^\vee) B(\bar{Z}^{v}(s)^\vee)\|_{ H_\gamma }^2\ud s\\\notag
&\le T \int_{0}^{T} \sup_{r\in[\tau_1,T]}\|(S_{\psi(\epsilon)}(r)-\bar{S}(r)^\vee) B(\bar{Z}^{v}(s)^\vee)\|_{ H_\gamma }^2\ud s+C\tau_1,~ \textup{a.s.}
\end{align}
Referring to Assumption \ref{asp:S-S} and \eqref{eq:deltaepsilon}, for any fixed $s\in[0,T]$,
$$\lim_{\epsilon\to 0}\sup_{r\in[\tau_1,T]}\|(S_{\psi(\epsilon)}(r)-\bar{S}(r)^\vee) B(\bar{Z}^{v}(s)^\vee)\|_{ H_\gamma } = 0,\quad \textup{a.s.}$$
In view of Assumption \ref{Asp:SG-strong}, \eqref{Stbound}, and \eqref{eq:BZ}, we have that for any $s\in[0,T]$,
$$\sup_{r\in[\tau_1,T]}\|(S_{\psi(\epsilon)}(r)-\bar{S}(r)^\vee) B(\bar{Z}^{v}(s)^\vee)\|_{ H_\gamma } \le C(N,\| \mathfrak{u} \|_{ H_\gamma }, T),\quad \textup{a.s.}$$
This, together with \eqref{eq:I22} and the bounded convergence theorem, yields that
\begin{align}\label{eq:I22ep}
&\limsup_{\epsilon\to 0}\E\bigg[\sup_{t\in[\tau_1,T]}\|I_{\epsilon,2,2}(t)\|_{ H_\gamma }^2\bigg]\\\notag
	&\le T\limsup_{\epsilon\to 0}\int_{0}^{T} \E\bigg[\sup_{r\in[\tau_1,T]}\|(S_{\psi(\epsilon)}(r)-\bar{S}(r)^\vee) B(\bar{Z}^{v}(s)^\vee)\|_{ H_\gamma }^2\bigg]\ud s+C\tau_1\\\notag
	&\le C\tau_1.
\end{align}

\textit{\bf Estimate of $I_{\epsilon,3,2}$}.
By the Cauchy--Schwarz inequality, using the fact that \newline $\|\cdot\|_{\mathcal{L}(\mathcal{U}_0,H_\gamma)}\le \|\cdot\|_{\mathcal{L}_2(\mathcal{U}_0,H_\gamma)}$, and $v^\epsilon\in \mathcal{A}_N$, we have that for any $t\in[0,T]$,
\begin{align*}
\|I_{\epsilon,3,2}(t)\|^2_{ H_\gamma}
&\le N\int_0^t \|(S_{\psi(\epsilon)}(t-s)- \bar{S}(t-s)^\vee)G(\bar{Z}^{v}(s)^\vee)\|^2_{\mathcal{L}_2(\mathcal{U}_0, H_\gamma )}\ud s
\end{align*}

Similar to the proof of \eqref{eq:I22}, we also have that for any $t\in[\tau_1,T]$,
\begin{align*}
\|I_{\epsilon,3,2}(t)\|^2_{ H_\gamma }
&\le C\int_{0}^{T} \sup_{r\in[\tau_1,T]}\|(S_{\psi(\epsilon)}(r)- \bar{S}(r)^\vee)G(\bar{Z}^{v}(s)^\vee)\|^2_{\mathcal{L}_2(\mathcal{U}_0, H_\gamma )}\ud s+C\tau_1,\quad \textup{a.s.}
\end{align*}
Recalling that $\{e_i\}_{i=1}^\infty$ is an orthonormal basis of $\mathcal{U}_0$, one has that for any $\eta\in \mathbb{N}^+$,
\begin{align}\label{eq:I32}\notag
\sup_{t\in[\tau_1,T]}\|I_{\epsilon,3,2}(t)\|^2_{ H_\gamma }
&\le C\sum_{i=1}^\eta\int_0^{T} \sup_{r\in[\tau_1,T]}\|(S_{\psi(\epsilon)}(r)- \bar{S}(r)^\vee)G(\bar{Z}^{v}(s)^\vee)e_i\|^2_{ H_\gamma }\ud s\\\notag
&\!\!\!\!\!\!\!\!\!+C\sum_{i=\eta+1}^{\infty}\int_0^{T} \sup_{r\in[\tau_1,T]}\|(S_{\psi(\epsilon)}(r)- \bar{S}(r)^\vee)G(\bar{Z}^{v}(s)^\vee)e_i\|^2_{ H_\gamma }\ud s+C\tau_1\\\notag
&\le C\sum_{i=1}^\eta\int_0^{T} \sup_{r\in[\tau_1,T]}\|(S_{\psi(\epsilon)}(r)- \bar{S}(r)^\vee)G(\bar{Z}^{v}(s)^\vee)e_i\|^2_{ H_\gamma }\ud s\\
&\quad+C\sum_{i=\eta+1}^{\infty}\int_0^{T}\|G(\bar{Z}^{v}(s)^\vee)e_i\|^2_{ H_\gamma }\ud s+C\tau_1,\quad \textup{a.s.}
\end{align}
It follows from the linear growth of $G$ (see \eqref{eq:Glingr}) and \eqref{eq:Zvt} that for $v\in\mathcal{A}_N$,
\begin{equation}\label{eq:GZ}
	\sup_{t\in[0,T]}\|G(\bar{Z}^{v}(s)^\vee)\|_{\mathcal{L}_2(\mathcal{U}_0, H_\gamma )}\le C(N,\| \mathfrak{u} \|_{ H_\gamma }, T),\quad \textup{a.s.}
\end{equation}
According to Assumption \ref{asp:S-S} and \eqref{eq:deltaepsilon}, for each fixed $i\in\mathbb{N}^+$ and $s\in[0,T]$,
 $$ \lim_{\epsilon\to 0}\sup_{r\in[\tau_1,T]}\|(S_{\psi(\epsilon)}(r)- \bar{S}(r)^\vee)G(\bar{Z}^{v}(s)^\vee)e_i\|^2_{ H_\gamma }= 0,\quad \textup{a.s.}.$$
Moreover, by Assumption \ref{Asp:SG-strong}, \eqref{Stbound}, and \eqref{eq:Zvt}, it holds that for any $s\in[0,T]$ and $i\in\mathbb{N}^+$,
\begin{align*} \sup_{r\in[\tau_1,T]}\|(S_{\psi(\epsilon)}(r)- \bar{S}(r)^\vee)G(\bar{Z}^{v}(s)^\vee)e_i\|^2_{ H_\gamma }&\le C\sup_{t\in[0,T]}\|G(\bar{Z}^{v}(s)^\vee)\|_{\mathcal{L}_2(\mathcal{U}_0, H_\gamma )}\\
&\le C(N,\| \mathfrak{u} \|_{ H_\gamma }, T),\quad \textup{a.s.}\end{align*}
Applying the bounded convergence theorem, we conclude that for any fixed $\eta\in \mathbb{N}^+$,
\begin{equation}\label{eq:ESS}
\lim_{\epsilon\to 0}\sum_{i=1}^\eta\E\int_0^T \sup_{r\in[\tau_1,T]}\|(S_{\psi(\epsilon)}(r)- \bar{S}(r)^\vee)G(\bar{Z}^{v}(s)^\vee)e_i\|^2_{ H_\gamma }\ud s=0.
\end{equation}
Taking expectations on both sides of \eqref{eq:I32} and subsequently passing to the limit superior as $\epsilon\to 0$, we infer from \eqref{eq:ESS} that for any $\eta\in\mathbb{N}^+$,
\begin{equation}\label{eq:I32ep-new}
\limsup_{\epsilon\to 0}\E\bigg[\sup_{t\in[\tau_1,T]}\|I_{\epsilon,3,2}(t)\|^2_{ H_\gamma }\bigg]\le C\tau_1+C\sum_{i=\eta+1}^{\infty}\E\int_0^{T}\|G(\bar{Z}^{v}(s)^\vee)e_i\|^2_{ H_\gamma }\ud s.
\end{equation}
Thanks to \eqref{eq:GZ}, the series
 $\sum_{i=1}^{\infty}\E\int_0^{T}\|G(\bar{Z}^{v}(s)^\vee)e_i\|^2_{ H_\gamma }\ud s <\infty$. Hence, taking the limit as $\eta\to \infty$ on both sides of \eqref{eq:I32ep-new} yields
 \begin{equation}\label{eq:I32ep}
\limsup_{\epsilon\to 0}\E\bigg[\sup_{t\in[\tau_1,T]}\|I_{\epsilon,3,2}(t)\|^2_{ H_\gamma }\bigg]\le C\tau_1.
\end{equation}

\textit{\bf Estimate of $I_{\epsilon,3,3}$}.
Taking Lemma \ref{lemma:finite}(2) and \eqref{fvarphi} into account, since $v^\epsilon$ converges weakly to $v$ a.s.,
$I_{\epsilon,3,3}$ converges to $0$ in $\mathcal{C}([0,T]; H_\gamma )$ as $\epsilon\to 0$, a.s. Besides, from \eqref{Stbound}, \eqref{eq:GZ} and $v^\epsilon,v\in\mathcal{A}_N$, we have that for any $t\in[0,T]$,
\begin{align*}
	\|I_{\epsilon,3,3}(t)\|_{ H_\gamma }^2&\le \int_0^t\| \bar{S}(t-s)^\vee G(\bar{Z}^{v}(s)^\vee)\|^2_{\mathcal{L}_2(\mathcal{U}_0, H_\gamma )}\ud s\int_0^t(\|v^\epsilon(s)\|_{\mathcal{U}_0}+\|v(s)\|_{\mathcal{U}_0})^2\ud s\\
	&\le C(N,\| \mathfrak{u} \|_{ H_\gamma }, T),\quad \textup{a.s.}
\end{align*}
By the bounded convergence theorem,
\begin{equation}\label{eq:I33ep}
\limsup_{\epsilon\to 0}\E\bigg[\sup_{t\in[\tau_1,T]}\|I_{\epsilon,3,3}(t)\|^2_{ H_\gamma }\bigg]\le \E\left[\limsup_{\epsilon\to 0}\|I_{\epsilon,3,3}\|_{\mathcal{C}([0,T]; H_\gamma )}^2\right]= 0.
\end{equation}

Inserting \eqref{eq:I1ep}, \eqref{eq:I22ep}, \eqref{eq:I32ep}, and \eqref{eq:I33ep} into \eqref{eq:Errwz} yields
 that for any $\tau_1\in(0,T)$,
\begin{align*}
\limsup_{\epsilon\to 0}\E\bigg[\sup_{t\in[\tau_1,T]}\|w_\epsilon^{v^\epsilon}(t)-\bar{Z}^{v}(t)^\vee\|_{ H_\gamma }^2\bigg]\le C(N,\| \mathfrak{u} \|_{ H_\gamma }, T)\tau_1.
\end{align*}
Furthermore, for any fixed $\tau_0\in(0,T]$, it holds that for any $\tau_1\in(0,\tau_0)$,
\begin{align*}
&\limsup_{\epsilon\to 0}\E\bigg[\sup_{t\in[\tau_0,T]}\|w_\epsilon^{v^\epsilon}(t)-\bar{Z}^{v}(t)^\vee\|_{ H_\gamma }^2\bigg]\\
&\le \limsup_{\epsilon\to 0}\E\bigg[\sup_{t\in[\tau_1,T]}\|w_\epsilon^{v^\epsilon}(t)-\bar{Z}^{v}(t)^\vee\|_{ H_\gamma }^2\bigg]\le C(N)\tau_1.
\end{align*}
Finally, letting $\tau_1\to 0^+$ completes the proof.
\end{proof}

\textbf{Proof of Theorem \ref{thm:LDP}.}
To prove Theorem \ref{thm:LDP}, let $\tau_0\in(0,T)$ be arbitrarily fixed and we define $\mathcal{G}^\epsilon=\textup{restr}_{[\tau_0,T]}\circ\mathcal{F}^\epsilon$ for any $\epsilon\ge0$, where 
the restriction map $\textup{restr}_{[\tau_0,T]}$ is given by $$\textup{restr}_{[\tau_0,T]}:\mathcal{C}([0,T];H_\gamma)\to \mathcal{C}([\tau_0,T];H_\gamma),\quad \Phi\mapsto \Phi\mid_{[\tau_0,T]}.$$ 
In light of Proposition \ref{prop:criterionMSZ} with \textsl{Condition A} and $\zeta(\epsilon)=\epsilon$, the proof of Theorem \ref{thm:LDP} boils down to showing that the following conditions \textbf{(A3)} and \textbf{(A4)} hold. 
\begin{enumerate}
\item[\textbf{(A3)}] For every $N >0$, the set $\left\{\mathcal G^0(\int_0^\cdot\phi(s)\ud s): \phi\in S_{N}\right\}$ is a compact subset of $\mathcal{C}([\tau_0,T];H_\gamma)$;

\item[\textbf{(A4)}] For every $N >0$, if $\{v^{\epsilon}\}_{\epsilon>0} \subset \mathcal{A}_{N}$ converges in distribution (as $S_{N}$-valued random elements) to $v$, then, as $\mathcal{C}([\tau_0,T];H_\gamma)$-valued random elements, $$\mathfrak{L}_{\epsilon}^{v^\epsilon}\rightarrow(\bar{Z}^v(\cdot))^\vee\quad \mbox{in distribution as } \epsilon\to 0.$$

\end{enumerate}

Recall that the map $S_N\ni \phi\mapsto (\bar{Z}^\phi)^\vee\in \mathcal{C}([0,T]; H_\gamma )$ is continuous (see Lemma \ref{lem:compact}). Given that $S_N$ endowed with the weak topology is compact,
the set $\{(\bar{Z}^\phi(\cdot))^\vee=\mathcal{F}^{0}\left(\int_{0}^\cdot \phi(s) \mathrm{d} s\right): \phi\in S_{N}\}$ is a compact subset of $\mathcal{C}([0,T]; H_\gamma )$ for any $N >0$. This verifies \textbf{(A3)}, as $\mathcal{G}^{0}\left(\int_{0}^\cdot \phi(s) \mathrm{d} s\right)$ is the restriction of $\mathcal{F}^{0}\left(\int_{0}^\cdot \phi(s) \mathrm{d} s\right)$ to the interval $[\tau_0,T]$. 
Let $\{v^\epsilon\}_{\epsilon>0}\subset\mathcal{A}_N$ converge to $v$ in distribution (as $S_N$-valued random elements) as $\epsilon\to 0$. Then by the Skorohod representation theorem, there exists a probability space $(\bar{\Omega},\bar{\mathcal{F}},\bar{\mathbb{P}})$ and $(\bar{v}^\epsilon, \bar{v})$ defined on $(\bar{\Omega},\bar{\mathcal{F}},\bar{\mathbb{P}})$ such that $\bar{v}^\epsilon=v^\epsilon $, $\bar{v}=v$ in distribution and $\bar{v}^\epsilon$ converges to $\bar{v}$ in $S_N$, $\bar{\mathbb{P}}$-a.s.
Recall that the maps $\phi\mapsto w_{\epsilon}^{\phi}$ and $\phi\mapsto (\bar{Z}^{\phi})^\vee$ are continuous from $L^2(0,T;\mathcal{U}_0)$ to $\mathcal{C}([\tau_0,T]; H_\gamma )$, and thus they are also measurable. This implies that $w_{\epsilon}^{v^\epsilon}=w_{\epsilon}^{\bar{v}^\epsilon}$ and $(\bar{Z}^{v})^\vee=(\bar{Z}^{\bar{v}})^\vee$ in distribution.
According to Lemma \ref{lem:w-Z}, for any fixed $\tau_0\in(0,T)$,
$w_{\epsilon}^{\bar{v}^\epsilon}$ converges to $(\bar{Z}^{\bar{v}})^\vee$ in $L^2(\bar{\Omega};\mathcal{C}([\tau_0,T]; H_\gamma ))$, and thus $w_{\epsilon}^{v^\epsilon}$ converges to $(\bar{Z}^{v})^\vee$ in distribution as $\mathcal{C}([\tau_0,T]; H_\gamma )$-valued random variables.
It then follows from Lemma \ref{lem:w-u} and Slutsky's theorem that $\mathfrak{L}_{\epsilon}^{v^\epsilon}$ converges to $(\bar{Z}^{v})^\vee$ in distribution as $\epsilon\to 0$, and hence proves \textbf{(A4)}. 
\hfill$\square$

\subsection{Proof of Theorem \ref{thm:MDP-graph}}
In this subsection, we present the proof of Theorem \ref{thm:MDP-graph} which establishes the LDP of $\{(\bar{X}^\epsilon)^\vee\}_{\epsilon\in(0,1]}$ on $\mathcal{C}([0,T];H_\gamma)$ as $\epsilon\to 0$.
By the definition in \eqref{eq:X}, for each $\epsilon>0$,
\begin{align}\label{eq:barXepsilon}
	\ud \bar{X}^\epsilon(t)&=\bar{L}\bar{X}^\epsilon(t)\ud t+\frac{B(\bar{u}^0(t)+\sqrt{\epsilon}\lambda(\epsilon)\bar{X}^\epsilon(t))-B(\bar{u}^0(t))}{\sqrt{\epsilon}\lambda(\epsilon)}\ud t
	\\\notag
	&\quad+\frac{G(\bar{u}^0(t)+\sqrt{\epsilon}\lambda(\epsilon)\bar{X}^\epsilon(t))}{\lambda(\epsilon)}\ud \bar{\mathcal{W}}(t)
\end{align}
for $t\in(0,T]$
with the initial value $\bar{X}^\epsilon(0)=0$.
For $\epsilon>0$,
let $\mathcal G^\epsilon$ be the measurable map associating $\mathcal{W}$ to $(\bar{X}^\epsilon)^\vee$, i.e., $(\bar{X}^\epsilon)^\vee=\mathcal G^\epsilon(\mathcal{W})$. 
For any $v\in\mathcal A_N$ and $\epsilon>0$, the Girsanov theorem indicates that $\widetilde{\mathcal{W}}^{\lambda(\epsilon)^{-2},v}:=\mathcal{W}+\lambda(\epsilon)\int_0^\cdot v(s)\ud s$ is a $\mathcal{U}_0$-cylindrical Wiener process under $\widetilde{\mathbb{P}}^{\lambda(\epsilon)^{-2},v}$ (see \eqref{eq:P}).
Then
$\mathcal G^\epsilon(\widetilde{\mathcal{W}}^{\lambda(\epsilon)^{-2},v})=(\bar{\mathfrak{M}}_{\epsilon}^v)^\vee$, where $\bar{\mathfrak{M}}_{\epsilon}^v$ is the unique mild solution to \eqref{eq:barXepsilon} with $(\bar{X}^\epsilon,\mathcal{W})$ replaced by $(\bar{\mathfrak{M}}_{\epsilon}^v,\widetilde{\mathcal{W}}^{\lambda(\epsilon)^{-2},v})$, under $\widetilde{\mathbb{P}}^{\lambda(\epsilon)^{-2},v}$. Since $\mathbb{P}$ is equivalent to $\widetilde{\mathbb{P}}^{\lambda(\epsilon)^{-2},v}$, $\bar{\mathfrak{M}}_{\epsilon}^v$ is also the unique mild solution to the following stochastic controlled equation associated with \eqref{eq:barXepsilon}
\begin{align} \label{eq.controlledX-eq}
	\ud& \bar{\mathfrak{M}}_{\epsilon}^v(t)
	=\bar{L}\bar{\mathfrak{M}}_{\epsilon}^v(t)\ud t+\frac{B(\bar{u}^0(t)+\sqrt{\epsilon}\lambda(\epsilon)\bar{\mathfrak{M}}_{\epsilon}^v(t))-B(\bar{u}^0(t))}{\sqrt{\epsilon}\lambda(\epsilon)}\ud t\\\notag &\quad+\frac{G(\bar{u}^0(t)+\sqrt{\epsilon}\lambda(\epsilon)\bar{\mathfrak{M}}_{\epsilon}^v(t))}{\lambda(\epsilon)}\ud \bar{\mathcal{W}}(t)+G(\bar{u}^0(t)+\sqrt{\epsilon}\lambda(\epsilon)\bar{\mathfrak{M}}_{\epsilon}^v(t))v(t)^\wedge\ud t 
\end{align}
for $t\in[0,T]$
with $\bar{\mathfrak{M}}_{\epsilon}^v(0) = 0$, under $\mathbb{P}$.
Passing to the limit as $\epsilon\to 0$ in \eqref{eq.controlledX-eq} and utilizing \eqref{eq:MDPpara} yield the skeleton equation \eqref{eq:Nskeleton}.
Hence, we denote 
\begin{equation}\label{eq:G0MDP}
	\mathcal G^0\left(\int_0^\cdot \phi(s)\ud s\right):=(\bar{R}^\phi(\cdot))^\vee,\quad \phi\in L^2(0,T;\mathcal{U}_0).
\end{equation}

\textbf{Proof of Theorem \ref{thm:MDP-graph}.}
We prove Theorem \ref{thm:MDP-graph} by applying Proposition \ref{prop:criterionMSZ} with \textsl{Condition A}${^\prime}$ and $\zeta(\epsilon)=\lambda(\epsilon)^{-2}$. Recalling that $\mathcal G^\epsilon(\widetilde{\mathcal{W}}^{\lambda(\epsilon)^{-2},v^\epsilon})=(\bar{\mathfrak{M}}_{\epsilon}^{v^\epsilon})^\vee$ and $\mathcal G^0$ defined in \eqref{eq:G0MDP}, it suffices to verify the following conditions \textbf{(A5)} and \textbf{(A6)}. 
\begin{enumerate}
	\item[\textbf{(A5)}] For every $N >0$, if $\{\phi^{\epsilon}\}_{\epsilon>0} \subset S_{N}$ converges weakly to $\phi$, then
	$\bar{R}^{\phi^\epsilon}\to \bar{R}^\phi$ in $\mathcal{C}([0,T];\bar{H}_\gamma)$ as $\epsilon\to 0$;
	\item[\textbf{(A6)}] For every $N >0$, any family $\{v^\epsilon\}_{\epsilon>0}\subset \mathcal{A}_N$ and any $\delta_0>0$,
	$$\lim_{\epsilon\to 0}\mathbb{P}\left\{\sup_{t\in[0,T]}\|\bar{\mathfrak{M}}_{\epsilon}^{v^\epsilon}(t)-\bar{R}^{v^\epsilon}(t)\|_{\bar{H}_\gamma}>\delta_0 \right\}=0.$$
\end{enumerate}
Condition \textbf{(A5)} has already been established in Lemma \ref{lem:compact}. 
Moreover, condition \textbf{(A6)} is a result of Lemma \ref{lem:M-Z}, \eqref{eq:MDPpara}, and the Markov inequality. 
\hfill$\square$

To prove Lemma \ref{lem:M-Z}, we prepare for regularity estimates of the skeleton equation \eqref{eq:Nskeleton} and the controlled equation \eqref{eq.controlledX-eq} in Lemmas \ref{lem:Rphibound} and \ref{lem:Mv}, respectively.

\begin{lemma}\label{lem:Rphibound}
 Assume that $b$ is continuously differentiable with bounded derivative and $g$ is globally Lipschitz continuous. 
 Let Assumption \ref{asp:gamma} hold and $\mathfrak{u}\in L^{p}(\R^2,\gamma^\vee\ud x)$ for some $p\ge 2$. Then there exists some constant $C$ depending on $N$ such that
\begin{align*}
\sup_{t\in[0,T]}\| \bar{R}^\phi(t)\|_{L^{p}(\Gamma,\ud\nu_\gamma)}\le C\left(1+\|\mathfrak{u}\|_{L^{p}(\R^2,\gamma^\vee\ud x)}\right)\quad \forall~\phi\in S_N.
\end{align*}
\end{lemma}
\begin{proof}
Notice that $\|\varphi^\wedge\|_{L^q(\Gamma,\ud \nu_\gamma)}\le \|\varphi\|_{L^{q}(\R^2,\gamma^\vee\ud x)}$ for any $\varphi\in L^{q}(\R^2,\gamma^\vee\ud x)$ and $q\ge1$.
Indeed, by \eqref{eq:wedge} and the H\"older inequality, 
\begin{align*}
\|\varphi^\wedge\|^q_{L^q(\Gamma,\ud \nu_\gamma)}
&=\sum_{k=1}^m\int_{I_k}\Big|\oint_{C_k(z)}\varphi(x)\mu_{z,k}\Big|^qT_k(z)\gamma(z,k)\ud z\\
&\le\sum_{k=1}^m\int_{I_k}\oint_{C_k(z)}|\varphi(x)|^q\mu_{z,k}T_k(z)\gamma(z,k)\ud z\\
&=\sum_{k=1}^m\int_{I_k}\oint_{C_k(z)}|\varphi(x)|^q\frac{1}{|\nabla \mathcal{H}(x)|}\gamma^\vee(x)\ud l_{z,k}\ud z=\|\varphi\|_{L^{q}(\R^2,\gamma^\vee\ud x)}^q.
\end{align*}

In view of \eqref{eq:utzk-deter}, Lemma \ref{lem:St}, the linear growth of $b$, and \eqref{eq:IkTk},
\begin{align*}
\| \bar{u}^0(t)\|_{L^p(\Gamma,\ud\nu_\gamma)}
&\le C\|\mathfrak{u}^\wedge\|_{L^p(\Gamma,\ud\nu_\gamma)}+C\int_0^t (1+\|\bar{u}^0(s)\|_{L^p(\Gamma,\ud\nu_\gamma)})\ud s.
\end{align*}
This allows us to apply the Gronwall inequality and deduce that
\begin{equation}\label{eq:u0}
\sup_{t\in[0,T]}\| \bar{u}^0(t)\|_{L^p(\Gamma,\ud\nu_\gamma)}\le C(1+\|\mathfrak{u}^\wedge\|_{L^p(\Gamma,\ud\nu_\gamma)})\le C(1+\|\mathfrak{u}\|_{L^{p}(\R^2,\gamma^\vee\ud x)}).
\end{equation}
Similarly, by \eqref{eq:Nskeleton}, the boundedness of $b^\prime$, and the linear growth of $g$, one has
\begin{align*}
		&\|\bar{R}^{\phi}(t)\|_{L^p(\Gamma,\ud\nu_\gamma)}\\
		&\le C\int_0^t \|{\mathcal{D}B}(\bar{u}^0(s))\bar{R}^{\phi}(s) \|_{L^p(\Gamma,\ud\nu_\gamma)}\ud s+C\int_0^t \|G(\bar{u}^{0}(s)) \phi(s)^\wedge\|_{L^p(\Gamma,\ud\nu_\gamma)} \ud s\\
		&\le C\int_0^t \|\bar{R}^{\phi}(s) \|_{L^p(\Gamma,\ud\nu_\gamma)}\ud s+C\int_0^t (1+\|\bar{u}^{0}(s) \|_{L^p(\Gamma,\ud\nu_\gamma)}) \|\phi(s)^\wedge\|_{L^\infty(\Gamma,\ud\nu_\gamma)}\ud s.
	\end{align*}
	In view of $\phi\in S_N\subset L^2(0,T;\mathcal{U}_0)$,
	 the inequality \eqref{eq:eiwedge} implies that for almost everywhere $(s,z,k)\in[0,T]\time\gamma\times\Gamma$,
	\begin{align*}|\phi(s)^\wedge(z,k)|&=\left|\sum_{i=1}^\infty\langle \phi(s),e_i\rangle_{\mathcal{U}_0}e_i^\wedge(z,k)\right|\\
	&\le \left(\sum_{i=1}^\infty\langle \phi(s),e_i\rangle_{\mathcal{U}_0}^2\right)^{\frac12} \left(\sum_{i=1}^\infty|e_i^\wedge(z,k)|^2\right)^{\frac12}\le C\|\phi(s)\|_{\mathcal{U}_0}.
	\end{align*}
Furthermore, for $\phi\in S_N$, using the Cauchy--Schwarz inequality, we have
\begin{align*}
		\|\bar{R}^{\phi}(t)\|_{L^p(\Gamma,\ud\nu_\gamma)}^2
		&\le C\int_0^t \|\bar{R}^{\phi}(s) \|_{L^p(\Gamma,\ud\nu_\gamma)}^2\ud s+C(N)\int_0^t (1+\|\bar{u}^{0}(s) \|_{L^p(\Gamma,\ud\nu_\gamma)}^2) \ud s.
	\end{align*}
		Finally, using \eqref{eq:u0} and applying the Gronwall inequality, we finish the proof.
	\end{proof}
	
	\begin{lemma}\label{lem:Mv}
 Assume that $b$ is continuously differentiable with bounded derivative and $g$ is globally Lipschitz continuous. 
 Let Assumption \ref{asp:gamma} hold and $\mathfrak{u}\in H_\gamma$.
Then for any $p\ge1$,
\begin{equation*}
	\E\left[\|\bar{\mathfrak{M}}_{\epsilon}^v(t)\|_{ H_\gamma }^p\right]\le C(T,p,\| \mathfrak{u} \|_{ H_\gamma },N)\quad\forall~\epsilon\in(0,1],t\in[0,T],v\in \mathcal{A}_N.
\end{equation*}
\end{lemma}
\begin{proof}
The proof follows by standard arguments using \eqref{eq.controlledX-eq}, \eqref{Stbound}, and \eqref{eq:u0}, and the details are thus omitted.
\end{proof}

Based on Lemmas \ref{lem:Rphibound} and \ref{lem:Mv}, we are able to measure the mean square error between the skeleton equation \eqref{eq:Nskeleton} and the controlled equation \eqref{eq.controlledX-eq}; see Appendix \ref{App:A2} for its proof.

\begin{lemma}\label{lem:M-Z}
 Assume that $b$ is continuously differentiable with the derivative $b^\prime$ being $\alpha_0$th H\"older continuous for some $\alpha_0\in(0,1]$, and $g$ is globally Lipschitz continuous.
Let Assumption \ref{asp:gamma} hold and $\mathfrak{u}\in L^{2\alpha_0+2}(\R^2,\gamma^\vee\ud x)$.
Then for any $\{v^\epsilon\}_{\epsilon\in(0,1]}\subset \mathcal{A}_N$ with some $N >0$, \begin{equation*}
\E\bigg[\sup_{t\in[0,T]}\|\bar{\mathfrak{M}}_{\epsilon}^{v^\epsilon}(t)-\bar{R}^{v^\epsilon}(t)\|^2_{\bar{H}_\gamma }\bigg]\le C\Big((\sqrt{\epsilon}\lambda(\epsilon))^{2\alpha_0}+\lambda(\epsilon)^{-2}\Big).
\end{equation*}
\end{lemma}

\appendix
\section{Proofs of lemmas}

\subsection{Proof of Lemma \ref{lem:St}}\label{App:A1}
For $f\in L^q(\Gamma,\ud \nu_\gamma)$,
we denote $\Upsilon_f(t,z,k):=(\bar{S}(t) f)(z,k)$ for $(t,z,k)\in[0,T]\times \Gamma$. 
{It} follows from \eqref{eq:barL}, the integration by parts formula, and the gluing condition \eqref{eq:glue} that for any $p\in \mathbb{N}^+$, 
\begin{align*}
&\frac{\ud}{\ud t}\sum_{k=1}^m\int_{I_k}|\Upsilon_f(t,z,k)|^{2p}T_k(z)\gamma(z,k)\ud z\\
&=p\sum_{k=1}^m\int_{I_k}\Upsilon_f(t,z,k)^{2p-1}\frac{\partial}{\partial z}\left(\alpha_k(z) \frac{\partial \Upsilon_f(t,z,k)}{\partial z}\right) \gamma(z,k)\ud z\\
&=-p(2p-1)\sum_{k=1}^m\int_{I_k}\Upsilon_f(t,z,k)^{2p-2}\Big|\frac{\partial}{\partial z}\Upsilon_f(t,z,k)\Big|^2\gamma(z,k)\alpha_k(z) \ud z\\
&\quad-p\sum_{k=1}^m\int_{I_k}\Upsilon_f(t,z,k)^{2p-1}\alpha_k (z)\frac{\partial \Upsilon_f(t,z,k)}{\partial z}\frac{\ud }{\ud z}\gamma(z,k)\ud z.
\end{align*}
By the Cauchy--Schwarz inequality,
\begin{align*}
&\left|\sum_{k=1}^m\int_{I_k}\Upsilon_f(t,z,k)^{2p-1}\alpha_k (z)\frac{\partial \Upsilon_f(t,z,k)}{\partial z}\frac{\ud }{\ud z}\gamma(z,k)\ud z\right|\\
&\le\frac12\sum_{k=1}^m\int_{I_k}\Upsilon_f(t,z,k)^{2p-2}\Big|\frac{\partial}{\partial z}\Upsilon_f(t,z,k)\Big|^2\gamma(z,k)\alpha_k(z)\ud z\\
&\quad+\frac12\sum_{k=1}^m\int_{I_k}\Upsilon_f(t,z,k)^{2p}\alpha_k (z)|\frac{\ud }{\ud z}\gamma(z,k)|^2\gamma^{-1}(z,k)\ud z.
\end{align*}
 Furthermore, using Assumption \ref{asp:gamma}, it holds that
\begin{align*}
&\frac{\ud}{\ud t}\sum_{k=1}^m\int_{I_k}|\Upsilon_f(t,z,k)|^{2p}T_k(z)\gamma(z,k)\ud z\\
&\le -p(2p-\frac32)\sum_{k=1}^m\int_{I_k}\Upsilon_f(t,z,k)^{2p-2}\Big|\frac{\partial}{\partial z}\Upsilon_f(t,z,k)\Big|^2\gamma(z,k)\alpha_k(z) \ud z\\
&\quad+C\sum_{k=1}^m\int_{I_k}|\Upsilon_f(t,z,k)|^{2p}T_k(z)\gamma(z,k)\ud z.
\end{align*}
This, together with the Gronwall inequality, implies that for any $p\in \mathbb{N}^+$,
$$\sum_{k=1}^m\int_{I_k}|\Upsilon_f(t,z,k)|^{2p}T_k(z)\gamma(z,k)\ud z\le C\sum_{k=1}^m\int_{I_k}|f(t,z,k)|^{2p}T_k(z)\gamma(z,k)\ud z.$$
 This proves \eqref{eq:Stq} for $q=2p$ with $p\in\mathbb{N}^+$, which, along with the Riesz--Thorin interpolation theorem, completes the proof for general $q\ge2$.
 \hfill$\square$

\subsection{Proof of Lemma \ref{lem:M-Z}}\label{App:A2}
Combining \eqref{eq:Nskeleton} and
 \eqref{eq.controlledX-eq} results in
 \begin{align*}
	\ud (\bar{\mathfrak{M}}_{\epsilon}^{v^\epsilon}(t)-\bar{R}^{v^\epsilon}(t))&=\bar{L}(\bar{\mathfrak{M}}_{\epsilon}^{v^\epsilon}(t)-\bar{R}^{v^\epsilon}(t))\ud t \\
	&+\frac{B(\bar{u}^0(t)+\sqrt{\epsilon}\lambda(\epsilon)\bar{\mathfrak{M}}_{\epsilon}^{v^\epsilon}(t))-B(\bar{u}^0(t)+\sqrt{\epsilon}\lambda(\epsilon)\bar{R}^{v^\epsilon}(t))}{\sqrt{\epsilon}\lambda(\epsilon)}\ud t\\
	&+\left(\frac{B(\bar{u}^0(t)+\sqrt{\epsilon}\lambda(\epsilon)\bar{R}^{v^\epsilon}(t))-B(\bar{u}^0(t))}{\sqrt{\epsilon}\lambda(\epsilon)}-\mathcal{D}B(\bar{u}^0(t))\bar{R}^{v^\epsilon}(t)\right)\ud t\\
	&+\frac{1}{\lambda(\epsilon)}G(\bar{u}^0(t)+\sqrt{\epsilon}\lambda(\epsilon)\bar{\mathfrak{M}}_{\epsilon}^{v^\epsilon}(t))\ud \bar{\mathcal{W}}(t)\\
	&+\left(G(\bar{u}^0(t)+\sqrt{\epsilon}\lambda(\epsilon)\bar{\mathfrak{M}}_{\epsilon}^{v^\epsilon}(t))-G(\bar{u}^{0}(t)) \right){v^\epsilon}(t)^\wedge\ud t.
\end{align*}
By \eqref{Stbound}, the Lipschitz continuity of $B$ and $G$ in \eqref{eq:Glip}, and the mean value theorem, 
\begin{align*} \|\bar{\mathfrak{M}}_{\epsilon}^{v^\epsilon}(t)&-\bar{R}^{v^\epsilon}(t)\|_{\bar{H}_\gamma}\le C\int_0^t\|\bar{\mathfrak{M}}_{\epsilon}^{v^\epsilon}(s)-\bar{R}^{v^\epsilon}(s)\|_{\bar{H}_\gamma}\ud s 	\\\notag&\quad+C\int_0^t\int_0^1\left\|\left(\mathcal{D}B(\bar{u}^0(s)+\sqrt{\epsilon}\lambda(\epsilon)\theta \bar{R}^{v^\epsilon}(s))-\mathcal{D}B(\bar{u}^0(s))\right)\bar{R}^{v^\epsilon}(s)\right\|_{\bar{H}_\gamma} \ud \theta
\ud s\\
	&\quad+\frac{1}{\lambda(\epsilon)}\left\|\int_0^t \bar{S}(t-s)G(\bar{u}^0(s)+\sqrt{\epsilon}\lambda(\epsilon)\bar{\mathfrak{M}}_{\epsilon}^{v^\epsilon}(s))\ud \bar{\mathcal{W}}(s)\right\|_{\bar{H}_\gamma}\\
	&\quad+C\sqrt{\epsilon}\lambda(\epsilon)\int_0^t \|\bar{\mathfrak{M}}_{\epsilon}^{v^\epsilon}(s)\|_{\bar{H}_\gamma}\left\|{v^\epsilon}(s)\right\|_{\mathcal{U}_0}\ud s.
\end{align*}
Since $b^\prime$ is $\alpha_0$th H\"older continuous, Lemma \ref{lem:Rphibound} and \eqref{eq:B} yield that for any $\theta\in[0,1]$,
\begin{align*}
&\left\|\left(\mathcal{D}B(\bar{u}^0(s)+\sqrt{\epsilon}\lambda(\epsilon)\theta \bar{R}^{v^\epsilon}(s))-\mathcal{D}B(\bar{u}^0(s))\right)\bar{R}^{v^\epsilon}(s)\right\|_{\bar{H}_\gamma}^2\\
&\le C\left(\sqrt{\epsilon}\lambda(\epsilon)\right)^{2\alpha_0}\sum_{k=1}^m\int_{I_k}| \bar{R}^{v^\epsilon}(s,z,k)|^{2\alpha_0+2}\gamma(z,k)T_k(z)\ud z\\
&\le C\left(\sqrt{\epsilon}\lambda(\epsilon)\right)^{2\alpha_0}\left(1+\|\mathfrak{u}\|_{L^{2\alpha_0+2}(\R^2,\gamma^\vee\ud x)}^{2\alpha_0+2}\right),\quad \mathbb{P}\textup{-a.s.,}
\end{align*}
where the constant $C$ may depend on $N$ but is independent of $\omega\in\Omega$, $\theta\in[0,1]$ and $s\in[0,T]$. Therefore,
for any $v^\epsilon\in \mathcal{A}_N$ and $t_1\in(0,T]$,
\begin{align}\label{eq:M-R}
	&\E\left[\sup_{t\in[0,t_1]}\|\bar{\mathfrak{M}}_{\epsilon}^{v^\epsilon}(t)-\bar{R}^{v^\epsilon}(t)\|_{\bar{H}_\gamma}^2\right]\\\notag
	&\le C\int_0^{t_1}\E\left[\sup_{r\in[0,s]}\|\bar{\mathfrak{M}}_{\epsilon}^{v^\epsilon}(r)-\bar{R}^{v^\epsilon}(r)\|_{\bar{H}_\gamma}^2\right]\ud s 	+C(\sqrt{\epsilon}\lambda(\epsilon))^{2\alpha_0}\\\notag
	&\quad+\frac{1}{\lambda(\epsilon)^2}\E\left[\sup_{t\in[0,T]}\left\|\int_0^t \bar{S}(t-s)G(\bar{u}^0(s)+\sqrt{\epsilon}\lambda(\epsilon)\bar{\mathfrak{M}}_{\epsilon}^{v^\epsilon}(s))\ud \bar{\mathcal{W}}(s)\right\|_{\bar{H}_\gamma}^2\right]\\\notag
	&\quad+C(\sqrt{\epsilon}\lambda(\epsilon))^2N\int_0^{T}\E\left[ \|\bar{\mathfrak{M}}_{\epsilon}^{v^\epsilon}(s)\|_{\bar{H}_\gamma}^2\right]\ud s,
\end{align}
where $C>0$ is a constant that depends on $\mathfrak{u}$ but not on $\epsilon$.
By the factorization argument, it follows from Lemma \ref{lem:Mv}, \eqref{eq:u0} and the linear growth of $G$ in \eqref{eq:Glingr} that for any $q>2$,
\begin{gather}\label{eq:Gbound}
\E\left[\sup_{t\in[0,T]}\left\|\int_0^t \bar{S}(t-s)G(\bar{u}^0(s)+\sqrt{\epsilon}\lambda(\epsilon)\bar{\mathfrak{M}}_{\epsilon}^{v^\epsilon}(s))\ud \bar{\mathcal{W}}(s)\right\|^q_{\bar{H}_\gamma}\right]\le C.
\end{gather}
Inserting \eqref{eq:Gbound} and the estimate from Lemma \ref{lem:Mv} into \eqref{eq:M-R}, we infer that for any $v^\epsilon\in \mathcal{A}_N$ and $t_1\in(0,T]$,
\begin{align*}
	\E\left[\sup_{t\in[0,t_1]}\|\bar{\mathfrak{M}}_{\epsilon}^{v^\epsilon}(t)-\bar{R}^{v^\epsilon}(t)\|_{\bar{H}_\gamma}^2\right]&\le C\int_0^{t_1}\E\left[\sup_{r\in[0,s]}\|\bar{\mathfrak{M}}_{\epsilon}^{v^\epsilon}(r)-\bar{R}^{v^\epsilon}(r)\|_{\bar{H}_\gamma}^2\right]\ud s \\
	&\quad+C(\sqrt{\epsilon}\lambda(\epsilon))^{2\alpha_0}+C\lambda(\epsilon)^{-2}+C(\sqrt{\epsilon}\lambda(\epsilon))^2,
\end{align*}
where the constant $C>0$ is independent of $v^\epsilon\in \mathcal{A}_N$.
Finally, we complete the proof by applying the Gronwall inequality and using \eqref{eq:MDPpara}.
\hfill$\square$

\bibliographystyle{abbrv}
\bibliography{references}
\end{document}